\documentclass[a4paper]{amsart}
\usepackage[utf8]{inputenc}
\usepackage[T1]{fontenc}

\usepackage{etoolbox}

\usepackage{mathtools}
\usepackage{amsmath, amssymb, amsthm}
\usepackage{bm}

\usepackage{enumitem}

\usepackage[hidelinks]{hyperref}
\usepackage[capitalise]{cleveref}
\crefname{equation}{}{}
\Crefname{equation}{}{}

\numberwithin{environment}{section}
\numberwithin{equation}{section}

\newcommand{\NewTheoremEnv}[2]{%
  \newtheorem{#1}[environment]{#2}%
  \AtBeginEnvironment{#1}{\crefalias{environment}{#1}}%
}

\theoremstyle{plain}
\NewTheoremEnv{theorem}{Theorem}
\NewTheoremEnv{proposition}{Proposition}
\NewTheoremEnv{lemma}{Lemma}
\NewTheoremEnv{corollary}{Corollary}

\theoremstyle{definition}
\NewTheoremEnv{definition}{Definition}
\NewTheoremEnv{remark}{Remark}
\NewTheoremEnv{construction}{Construction}

\let\cal\mathcal
\AtBeginDocument{\def\vec#1{\boldsymbol{#1}}}
\newcommand{\norm}[1]{\lVert #1 \rVert}

\newcommand{\ZZ}{\mathbb{Z}}
\newcommand{\QQ}{\mathbb{Q}}
\newcommand{\RR}{\mathbb{R}}
\newcommand{\FF}{\mathbb{F}}
\newcommand{\PP}{\mathbb{P}}

\newcommand{\Dc}{\mathcal{D}}
\newcommand{\Ic}{\mathcal{I}}
\newcommand{\Pc}{\mathcal{P}}
\newcommand{\Yc}{\mathcal{Y}}

\DeclareMathOperator{\id}{id}
\DeclareMathOperator{\Br}{Br}
\DeclareMathOperator{\Cone}{Cone}
\DeclareMathOperator{\Pic}{Pic}
\DeclareMathOperator{\vol}{vol}
\DeclareMathOperator*{\argmax}{arg\,max}

\title[Everywhere locally soluble varieties over toric bases]{On everywhere locally soluble varieties \\ in families over toric bases}

\author[V. Mitankin]{Vladi Mitankin}
\address{Institute of Mathematics and Informatics, Bulgarian Academy of Sciences, Acad.\ G.\ Bonchev St., bl.\ 8, 1113 Sofia, Bulgaria}
\email{\href{mailto:v.mitankin@math.bas.bg}{v.mitankin@math.bas.bg}}

\author[L. Schäfer]{Leo Schäfer}
\address{Mathematisches Institut, Georg-August-Universität Göttingen, Bunsenstr.\ 3--5, 37073 Göttingen, Germany}
\email{\href{mailto:leo.schaefer@mathematik.uni-goettingen.de}{leo.schaefer@mathematik.uni-goettingen.de}}

\author[D. Schindler]{Damaris Schindler}
\email{\href{mailto:damaris.schindler@mathematik.uni-goettingen.de}{damaris.schindler@mathematik.uni-goettingen.de}}

\author[F. Wilsch]{Florian Wilsch}
\email{\href{mailto:florian.wilsch@mathematik.uni-goettingen.de}{florian.wilsch@mathematik.uni-goettingen.de}}

\dedicatory{To Roger Heath-Brown on the occasion of his 75th birthday}

\begin{document}

\begin{abstract}
   We determine the density of locally soluble fibres in certain families of varieties over a toric base. For this we develop a new form of the hyperbola method for toric bases which allows for counting functions with a growth behaviour of the form $B^a(\log B)^{-\Delta}$.
\end{abstract}

\subjclass[2020]{11D45 (11P21, 14D10, 14G05)}


\maketitle

\section{Introduction}

It is in general a very deep question to understand if a variety $X$ over the field of rational numbers $\mathbb{Q}$ has a rational point. Addressing that question for families of varieties has recently become an increasingly attractive topic of research. In families where the Hasse principle holds, members with a rational point are precisely those which are everywhere locally soluble. The study of varieties which are everywhere locally soluble has received a lot of attention on its own, as local solubility is typically much easier to understand. Once a suitable height function is introduced, it becomes natural to ask what the proportion of everywhere locally soluble varieties in specific families is.

The last question has been investigated for various families of conics by Hooley \cite{Hooley93}, Guo \cite{Guo}, Gamburd, Ghosh, Sarnak and Whang \cite{GGSW26}, Loughran \cite{Lou18}, Loughran, Rome and Sofos \cite{loughran2024leadingconstantrationalpoints}, Da Silva \cite{daSilva}, more general conic bundles by Destagnol, Lyczak and Sofos \cite{DLS25}, Fermat equations by Koymans, Paterson, Santens and Shute \cite{KPSS25}, more general families of quadrics by Wilson \cite{Wilson} and Bhargava, Cremona, Fisher, Jones and Keating \cite{BCFJK16}, systems of conics by Chan, Koymans and Rome \cite{CKR25}, Sofos \cite{Sof16}, del Pezzo surfaces by Mitankin and Salgado \cite{MS22}, hypersurfaces in products of projective spaces by Fisher, Ho and Park \cite{FHP21}, etc. For more references see the list provided in \cite{loughran2024leadingconstantrationalpoints}. In many cases, classical methods from analytic number theory are used, though in recent work of Gamburd, Ghosh, Sarnak and Whang \cite{GGSW26} methods from ergodic theory are applied to study the density of locally soluble conics.\par
When studying all these examples, a central leading question is in what way the density of locally soluble members in a certain family depends on geometric properties of the family. Loughran and Smeets \cite[Conj.~1.6]{LS16} set up a first conjecture in this direction for families of varieties which can be understood as fibres of a suitably nice dominant morphism $\pi\colon Y\rightarrow \mathbb{P}^n$. It is a very natural question to ask for more general base varieties than $\mathbb{P}^n$, assuming that the base variety has a sufficiently dense set of rational points. Among other questions, this has been taken up in \cite{loughran2024leadingconstantrationalpoints}. Consider a dominant morphism of smooth projective varieties $\pi\colon Y\rightarrow X$ over a number field $k$ with geometrically integral generic fibre. For an anticanonical height function $H\colon X(k)\rightarrow \mathbb{R}_{\geq 0}$ and an (exceptional) set $\Omega\subset X(k)$ one may define the counting function
\begin{align*}
 N_{\pi,\Omega}(B) = \sharp\{x\in X(k)\setminus\Omega: H(x)\leq B, x\in \pi(Y(\mathbb{A}_k))\},
\end{align*}
which counts the number of everywhere locally soluble fibres up to height $B$ outside of $\Omega$. Here we write $Y(\mathbb{A}_k)$ for the set of adelic points of $Y$ and $\Omega$ is typically taken to be a thin set. Loughran, Rome and Sofos~\cite{loughran2024leadingconstantrationalpoints} make several additional assumptions on the morphism $\pi$: they assume that it admits a smooth fibre over a rational point that is everywhere locally soluble and that the fibre over every codimension $1$ point of $X$ contains an irreducible component of multiplicity $1$.
Moreover, they assume that $X$ is weak Fano and that $\Br X = \ker (\Br X\rightarrow \Br X_{\bar k})$. Write $U\subset X$ for the open subset of $X$ arising by removing the closure of all codimension $1$ points of $X$ which lie below a non-split fibre, and assume that the rank $\rho(X)$ of the Picard group of $X$ satisfies $\rho(X)=1$ or that $k[U]^{\times}=k^\times$. They then conjecture that there exists a thin subset $\Omega \subset X(k)$ such that
\begin{align*}
N_{\pi,\Omega}(B) \sim c_\pi B (\log B)^{\rho(X)-\Delta(\pi)-1},
\end{align*}
where $\Delta(\pi)$ can be expressed as a (finite) sum over codimension $1$ points of $X$. Loughran, Rome and Sofos also give a precise prediction for the constant $c_\pi$ in analogy with Peyre's prediction of the leading constant in Manin's conjecture.

In this article we begin to explore this question, starting with the case of a smooth split toric variety $X$ over $\mathbb{Q}$ as a base variety. We obtain new families for which the asymptotic growth as predicted in the Loughran--Rome--Sofos conjecture holds. The assumption $\rho(X)=1$ or $k[U]^\times = k^{\times}$ is very restrictive. Relaxing that condition creates an interesting direction to explore the counting function $N_{\pi,\Omega}(B)$, and the class of families that we are treating includes ones that satisfy as well as ones that violate this assumption, though the conjecture holds for all of them.

We now start introducing the necessary notation so that we can state our main theorem precisely. In \cref{section:examples} we will give a number of concrete applications.

\subsection{Local solubility in families over toric bases}
Let $X$ be a smooth, $n$-dimensional, split toric variety over $\QQ$ of Picard rank $\rho\ge 2$, associated with a fan $\Sigma$. Let $\Dc\subset \Pic X$ be the set of classes of torus-invariant prime divisors $D$. As in the application~\cite{MR4959852} of the hyperbola method to subvarieties of toric varieties, we shall group the Cox ring generators by degree: for each class $d \in \Dc$, let $a_d$ be the number of invariant prime divisors with class $d$ in order to enumerate them $D_{d,1},\dots,D_{d,a_d}$. To each such class corresponds a ray $\rho_{d,j}$ of $\Sigma$ and a generator $x_{d,j}$ of the total coordinate ring
\begin{equation*}
	R = \QQ[x_{d,j} : d\in \Dc,\ 1\le j\le a_d].
\end{equation*}

\begin{construction}\label{constr:family-over-toric-base}
Let $\Ic \subseteq \Dc$ be a subset, and for each $i\in \Ic$, let $f_i$ be a homogeneous form in a set of variables $\Yc_i = \{y_{i,1},\dots, y_{i,m_i}\}$ with coefficients in $\QQ x_{i,1}\oplus \dots \oplus\QQ x_{i,a_i}$.
In particular, $f_i$ defines both a family
\begin{equation*}
	\pi_i\colon Y^{(i)}\subset L_i^{\oplus m_i}\to X,
\end{equation*}
where $L_i$ is a line bundle of class $i$,
and, more elementarily, a family
\begin{equation*}
	\psi_i\colon \PP^{m_i-1}\times \mathbb{A}^{a_i}  \supset V(f_i) \to \mathbb{A}^{a_i}
\end{equation*}
over affine space.
Assume that the number
\begin{equation*}
	N_i(B) = \sharp\{\vec x \in \mathbb{A}^{a_i}(\ZZ) : \psi_i^{-1}(\vec x)(\mathbb{A}_\QQ)\ne \emptyset,\ \Vert \vec x \Vert_\infty \le B,\ x_j \ne 0 \text{ for } 1\le j\le a_i\} 
\end{equation*}
of everywhere locally soluble fibres of the latter family satisfies
\begin{equation}\label{eq:affine-asymp}
	N_i(B) = c_i B^{a_i} (1+\log B)^{-\Delta_i} + O(B^{a_i}(1+\log B)^{-\Delta_i-\delta})
\end{equation}
with $\Delta_i\ge 0$, $c_i>0$, and $\delta>0$ for $B\geq 1$. Without loss of generality, we may assume $\delta<1/2$. We set $\Delta_d=0$ and $c_d=2^{a_d}$ for $d\in \Dc\setminus \Ic$ and, as there are no conditions on those variables, get an affine counting function
\begin{equation*}
	N_d(B) = \sharp\{\vec x \in \mathbb{A}^{a_d}(\ZZ) : \Vert \vec x \Vert_\infty \le B \} = c_d B^{a_d}(1+\log B)^{-\Delta_d} (1 +O(B^{-1}));
\end{equation*}
this can be thought of as being associated with the identity morphisms $\pi_d=\id_X$ and $\psi_d = \id_{\mathbb{A}^{a_d}}$.

Let
\begin{equation*}
	\pi\colon Y = \sideset{}{_X}\prod_{i\in \Ic} Y^{(i)} \to X
\end{equation*}
be their fibre product.
Our aim is to estimate the number $N(B)$ of everywhere locally soluble fibres over the open orbit in $X$.
\end{construction}

Our proof strategy consists in a combination of a universal torsor method and a hyperbola method. We now describe the resulting counting problem in terms of the universal torsor. The total coordinate space of $X$ is the complement $\mathbb{A}^{n+\rho}\setminus V(I)$ of the vanishing locus of the \emph{irrelevant ideal}
\begin{equation*}
	I =
	\prod_{\sigma\in \Sigma^{\max}}
	(x_{d,j} : \rho_{d,j}\not\subseteq \sigma) \subseteq R,
\end{equation*}
which induces a coprimality condition
\begin{equation}\label{eq:coprimality-toric}
	\gcd(x_{d,j} : \rho_{d,j}\not\subset\sigma) =1 \text{ for all } \sigma\in \Sigma^{\max}
\end{equation}
that appears in the description of the set
\begin{equation*}
	Y(\ZZ) = \{\vec x \in \ZZ^{n+\rho} : \text{\eqref{eq:coprimality-toric} holds}\}
\end{equation*}
of integral points on a universal torsor, parametrizing rational points on $X$ through a $2^\rho$-to-$1$ correspondence.
For every point $\vec x$ on this coordinate space and $i\in \Ic$, write $\vec x_i$ for the entries of $\vec x$ corresponding to $i$. Moreover, if
\begin{equation*}
	D = \sum_{\substack{d\in \Dc,\\1\le j\le a_d}} k_{d,j} D_{d,j}
\end{equation*}
is a torus-invariant effective divisor, write
\begin{equation*}
	\vec x^D = \prod_{\substack{d\in \Dc,\\1\le j\le a_d}} x_{d,j}^{k_{d,j}}.
\end{equation*}

We shall count points ordered by anticanonical height. In order to describe it, let, for every line bundle $L$ and maximal cone $\sigma\in\Sigma^{\max}$, $L(\sigma)$ denote the unique divisor that is linearly equivalent to $L$ and supported only on the divisors not corresponding to any ray in $\sigma$. In this notation,
\begin{equation*}
	H(\vec x) = \max_{\sigma\in \Sigma^{\max}} |\vec x^{-K(\sigma)}|,
\end{equation*}
where $-K = \sum_{d\in \Dc, 1\le j\le a_d} D_{d,j}$ is anticanonical. As we show in \cref{lem:factorization_H_theta}, the height $H\colon \ZZ^{n+\rho}\to\RR_{>0}$ factors through a height function $H'\colon \ZZ^{\Dc} \to \RR_{>0}$ with $H(\vec x)=H'(\norm{x_d}_\infty)_{d\in \Dc}$. Assume that $H'$ is given by $k$ monomials, that is, there are vectors $\vec\varpi_1,\dots,\vec\varpi_k\in \ZZ_{\ge 0}^\Dc$ such that
\begin{equation}\label{eq:package-height}
	H'(\vec b) = \max_{1\le j\le k}\vec b^{\vec\varpi_j} = \max_{1\le j\le k} \prod_{d\in \Dc} b_d^{\varpi_{j,d}}.
\end{equation}
From the data $\vec\varpi_1,\dots,\vec\varpi_k\in \ZZ_{\ge 0}^\Dc$, $(a_d)_{d\in \Dc}$ we obtain the polytope

\begin{equation*}
		\cal P
		= \Bigl\{\,\vec t \in \mathbb R_{\geq 0}^{|\Dc|} ~:~ \sum_{j\in \Dc} \frac{\varpi_{ij}}{a_j} t_j \leq 1 ~~\text{for}~~ i = 1, \dots, k\,\Bigr\}.
	\end{equation*}
Let $\cal F$ be the subset of $\cal P$ on which the function $|\vec t|_1$ is maximal, that is,
$\cal F= \argmax \limits_{\vec t \in \cal P} |\vec t|_1$. 
We say that $-\vec\Delta=(-\Delta_d)_{d\in\Dc}$ is \emph{of convergent type} if the integral 
\begin{equation}\label{eqn:convergenttypeF}
\int_{\cal F} \prod_{d\in\Dc} t_d^{-\Delta_d} d \mu
\end{equation}
is convergent, where $d\mu$ stands for the Lebesgue measure on the linear subspace spanned by the face $\cal F$. 

For the following theorem, these defining data --- the variety $X$, the equations $f_i$, the constants $c_d$, the implicit constants in~\eqref{eq:affine-asymp}, and the exponent $\delta$ --- are treated as fixed, and all implicit constants may depend on them. We write $\alpha(X)$ for Peyre's $\alpha$-constant and $\omega_\infty = \tau(X(\RR))$ for the real Tamagawa volume and, with this notation at hand, are ready to state our main theorem.

\begin{theorem}\label{main_theorem_toric}
	Let $X$ be a smooth, projective, split toric variety over $\QQ$ of Picard rank $\rho\ge 2$ with open orbit $U$, and let $\pi\colon Y\to X$ be as in \cref{constr:family-over-toric-base}. Write $\Delta = \sum_{d\in \Dc} \Delta_d$ and assume \cref{eq:affine-asymp} for $i\in \Ic$. If $-\vec \Delta$ is of convergent type, then the number 
	\begin{equation*}
		N(B) = \sharp\{x\in U(\QQ) : Y_x(\mathbb{A}_\QQ)\ne \emptyset,\ H(x)\le B\}
	\end{equation*}
	of everywhere locally soluble fibres satisfies
	\begin{equation*}
		N(B) = c B (\log B)^{\rho-1-\Delta} + O_\pi \!\left(B (\log B)^{\rho-1-\Delta-\varepsilon}\right),
	\end{equation*}
	where
	\begin{equation*}
		c= \frac{\alpha(X)}{2^{n+\rho}}  \omega_\infty \prod_p  \left( 1-\frac{1}{p}\right)^{\rho}\frac{|\mathfrak{X}(\FF_p)|}{p^n} \prod_{d\in \Dc} c_d a_d^{\Delta_d},
	\end{equation*}
    and $\varepsilon >0$.
\end{theorem}

A variety of examples which satisfy \cref{eq:affine-asymp} for various values of $\Delta_d$ can be found in the references at the beginning of this section. The remaining condition, to be of convergent type, can be made explicit fairly directly, and we shall do so for two classes of toric bases: those of Picard rank 2 and blow-ups of skew linear subspaces of projective space. It turns out that it is automatically verified if all $\Delta_i$ are smaller than $1$, which in turn follows from the Loughran--Rome--Sofos conjecture on the inputs~\cref{eq:affine-asymp} whenever $\QQ[U]^\times=\QQ^\times$, though our condition allows significantly more types of input (\cref{section:examples}).\par

It is an interesting question what happens in the case that $-\vec\Delta$ in \cref{main_theorem_toric} is not of convergent type, that is, when the corresponding integral in \cref{eqn:convergenttypeF} is not convergent. We plan to address this in future work. 
\subsection{The hyperbola method}
The proof of \cref{main_theorem_toric} relies on the development of a new form of the hyperbola method that allows for counting functions which have a growth behaviour as in~\cref{eq:affine-asymp}. This generalizes earlier forms of hyperbola methods, such as work of Blomer and Brüdern \cite{BloBru18} and Schäfer \cite{Schaefer}, which are both designed for special toric varieties such as products of projective spaces, the latter of which already allows for logarithmic terms to appear. It also generalizes work of Pieropan and Schindler \cite{PieSch24}, who treat general smooth, projective, split toric varieties over $\mathbb{Q}$, but who do not allow for any logarithmic terms to appear.

We need both the flexibility to work with multiple height conditions as well as the flexibility to allow for logarithmic terms in our counting function; hence, we provide here such a tool, which we expect to be useful for future applications as well. In \cref{section:hyperbola} we describe the precise set-up of our new form of the hyperbola method and state our main result in \cref{maintheorem} as well as present a proof of it. In \cref{sec:proofMT_examples} we prove \cref{main_theorem_toric} as well as give a number of concrete classes of examples to which it applies.\par

\subsection{Notation and conventions}
For a vector $\vec r\in \mathbb R^s$ we write $|\vec r| = \max_i |r_i|$ and $|\vec r|_1=\sum_{i=1}^s |r_i|$. For two vectors $\vec r, \vec s \in \mathbb{R}^s$ we write $\vec r \leq \vec s$ if $r_i\leq s_i$ for all $1\leq i\leq s$.
We write $\Delta^s= \{(t_1,\ldots, t_s)\in \mathbb{R}_{\geq 0}^s: |\vec t|_1\leq 1\}$ for the standard $s$-dimensional simplex, so that its dilate by $\kappa>0$ is $\kappa \Delta^s = \{(t_1,\ldots, t_s)\in \mathbb{R}_{\geq 0}^s: |\vec t|_1\leq \kappa\}$.
The parameters $s$, $k$, $\vec \alpha$, $\vec \ell$, $\vec\varpi$, $\delta$, and $A$ that appear in \cref{mainassumption} and \cref{upperboundassp} will be treated as fixed throughout the article, and all implicit constants may depend on them.

\subsection{Acknowledgements}
VM was supported by Horizon Europe 2023 MSCA postdoctoral fellowship 101151205 -- GIANT, funded by the European Union and by 2025 JESH scholarship of the Austrian Academy of Sciences.
FW was supported by the Deutsche Forschungsgemeinschaft (DFG) -- 398436923 (RTG2491).

\section{The hyperbola method}\label{section:hyperbola}

Let $f \colon \mathbb N^s \to \mathbb C$ be a function and $\vec\varpi_{i} = (\varpi_{i1}, \dots, \varpi_{is}) \in \mathbb R_{\geq 0}^s$ for $1\leq i \leq k$. Define
\begin{equation*}
	\Upsilon(B; \vec X) \coloneqq \sum_{\substack{\vec x \in \mathbb N^s \\ \prod_{j=1}^s x_j^{\varpi_{ij}} \leq B/X_i ~\forall 1\leq i\leq k}} f(\vec x).
\end{equation*}
Our goal is to obtain an asymptotic formula for $\Upsilon(B; \vec X)$ as $B \to \infty$, assuming that we can evaluate $f$ on boxes in a suitable sense and assuming that the parameters $X_i$ do not grow too fast with $B$. In this article we treat the case $X_i=1$, $1\leq i\leq k$, but for some of the preparations and lemmas we will already allow $X_i$, $1\leq i\leq k$, to be more general: we plan to use and apply those in future work in the context of counting varieties in families for which the Hasse principle or weak approximation fails due to a Brauer--Manin obstruction.

We assume that the function $f$ satisfies the following two properties: there exist constants $c_f, C_f\in \mathbb{C}$, $0 < \delta \leq 1$, $\vec\alpha \in \mathbb{R}_{>0}^s$ and $\vec \ell \in \mathbb{R}^s$ such that
\begin{multline}\label{mainassumption}
	\sum_{\vec x \leq \vec B} f(\vec x) = c_f \prod_{i=1}^s B_i^{\alpha_i} (1 + \log B_i)^{\ell_i}
	 \\ 
	+ O\Bigl( C_f \bigl(1 + \min_i \log B_i\bigr)^{-\delta} \prod_{i=1}^s B_i^{\alpha_i} (1 + \log B_i)^{\ell_i} \Bigr),
\end{multline}
with $|c_f| \leq C_f$. It will turn out that the sum $\ell \coloneq  \ell_1+\dots+\ell_s$ will feature in $\Upsilon$. Moreover, we assume that there exists an $A > 0$ such that for $1<\theta \leq 2$ we have
\begin{multline}\label{upperboundassp}
	\sum_{\vec B \leq \vec x \leq \theta \vec B} |f(\vec x)| \ll C_f \Bigl( (\log \theta)^s + (\log \theta)^{-A} (1 + \min_i \log B_i)^{-\delta} \Bigr) \\
	\times \prod_{i=1}^s B_i^{\alpha_i} (1 + \log B_i)^{\ell_i}.
\end{multline}
Note that if $f \colon \mathbb N^s \to \mathbb R_{\geq 0}$ takes non-negative values, then \cref{upperboundassp} is implied by \cref{mainassumption} and can be omitted.

We cover the region of integration in the summation for $\Upsilon(B; \vec X)$ by $\theta$-adic boxes. In order to decide which boxes belong to the summation range, we introduce the following polytope. 

\begin{definition}
	For $\vec \kappa \in \mathbb R_{> 0}^k$, $\vec \alpha = (\alpha_1, \dots, \alpha_s) \in \mathbb R_{> 0}^s$, and $\vec \varpi_i = (\varpi_{i1}, \dots, \varpi_{is}) \in \mathbb R_{\geq  0}^s$ for $i = 1, \dots, k$, define
	\begin{equation*}
		\cal P(\vec \kappa)
		= \Bigl\{\,\vec t \in \mathbb R_{\geq 0}^s ~:~ \sum_{j=1}^s \frac{\varpi_{ij}}{\alpha_j} t_j \leq \kappa_i ~~\text{for}~~ i = 1, \dots, k\,\Bigr\}.
	\end{equation*}
	Associated with this polytope, we define the \emph{maximal face} $\cal F(\vec \kappa)$ as the subset of $\cal P(\vec \kappa)$ on which the function $|\vec t|_1$ attains its maximum, that is, as
	\begin{equation*}
		\cal F(\vec \kappa) = \argmax \limits_{\vec t \in \cal P(\vec \kappa)} |\vec t|_1.
	\end{equation*}
	Since $|\vec t|_1$ is a linear function on $\cal P(\vec \kappa)$, the set $\cal F(\vec \kappa)$ is indeed a face and again a polytope. For simplicity, we denote $\cal P(\kappa) = \cal P(\kappa, \ldots, \kappa)$ and similarly $\cal F(\kappa) = \cal F(\kappa, \ldots, \kappa)$ for $\kappa \in \mathbb{R}_{> 0}$.

	Note that both $\cal P(\kappa)$ and $\cal F(\kappa)$ are homogeneous. We thus set $\cal P = \cal P (1)$ and $\cal F = \cal F(1)$ to get 
	\begin{equation}\label{eq:polytope-scaling}
		\cal P(\kappa) = \kappa \,\cal P
		\quad\text{and}\quad
		\cal F(\kappa) = \kappa \,\cal F.
	\end{equation}
\end{definition}
    
For the rest of the article we shall assume that the maximal face $\cal F$ is not contained in a coordinate hyperplane. We next need to make some more assumptions on the values of the $\ell_i$, $1\leq i\leq s$, that we allow in \cref{mainassumption}, depending on the shape of the polytope: an integral that will naturally appear as part of our leading constant has to be finite.

\begin{definition}
We say that a vector $\vec\ell\in \mathbb{R}^s$ is \emph{of convergent type} for the face $\cal F$ if the integral
\begin{equation}\label{eq:face-integral}
	\int_{\cal F} \prod_{i=1}^s t_i^{\ell_i} d \mu,
\end{equation}
where $d\mu$ stands for the $d$-dimensional Lebesgue measure on the face $\cal F$,
is finite.
\end{definition}

This property can be checked combinatorially: if the face is of dimension $d$, and $\cal E$ its set of vertices, let $\vec v\colon \{0, \ldots, d\} \to \cal E$ be an ordered tuple of vertices spanning a subset of full dimension. To state the criterion, define the auxiliary function
\begin{equation*}
    s(\vec v)_j = \begin{cases}
        \infty &\text{if}~v(0)_j \neq 0,\\
        \max\{\,i \,:\, v(i)_j \neq 0 \} &\text{if}~v(0)_j = 0.
    \end{cases}
\end{equation*}
This depends only on which coordinates of the $v(i)_j$ vanish.

\begin{proposition}\label{prop:convergence-test}
	A vector  $\vec\ell\in \mathbb{R}^s$ is of convergent type for the face $\cal F$ if and only if 
\begin{equation}\label{eq:l-sum-condition}
i +  \sum_{\substack {j \leq s \\[1pt] s(\vec v)_j\leq i}} \ell_j >0
\end{equation}
for all $\vec v\colon \{0,\dots,d\}\to \cal E$ spanning a subset of full dimension and all $1 \leq i \leq d$.
\end{proposition}

\begin{proof}
We write $\vec w(i) = \vec v(i) - \vec v(0)$. The integral~\eqref{eq:face-integral} converges if and only if its restriction to every full-dimensional simplex, spanned by vertices $\vec v(0),\dots,\vec v(d)$, does, that is, if
\begin{equation*}
	\int_{\Delta^d} \Bigl( (1 - |\vec \lambda|_1) \vec v(0) + \sum_{i=1}^d \lambda_i \vec v(i) \Bigr)^{\vec \ell} \,\mathrm d\vec \lambda
\end{equation*}
is finite.
Notice that these integrals are invariant under permutations of indices. Therefore, our face is still covered when we restrict the range of integration to $(1 - 1/(d+1)) \Delta^d$, since at least one of the $\lambda_i$ or $1 - |\vec \lambda|_1$ is at least $1/(d+1)$. This ensures that for any point $\vec \lambda \in (1 - 1/d) \Delta^d$ the $j$-th coordinate of 
\begin{equation*}
    (1 - |\vec \lambda|_1) \vec v(0) + \sum_{i=1}^d \lambda_i \vec v (i)
\end{equation*}
can only be zero if the $j$-th coordinate of $\vec v(0)$ is zero. Then the above integral is given by
\begin{equation*}
	\int_{(1 - 1/(d+1)) \Delta^d} \Bigl( \vec v(0) + \sum_{i=1}^d \lambda_i \vec w(i) \Bigr)^{\vec \ell} \,\mathrm d\vec \lambda.
\end{equation*}
We can further split the integrations by permuting the indices of $\lambda_i$ to ensure $\lambda_1 \leq \lambda_2 \leq \cdots \leq \lambda_d$. We have the bounds
\begin{equation*}
	v(0)_j + \sum_{i=1}^d \lambda_i w(i)_{j} \asymp_{d, \vec v(1), \dots, \vec v(d)} \max_i \begin{cases}
		1                                  & \text{if}~~ v(0)_{j} > 0, \\
		\max \{ \lambda_i : w(i)_{j} > 0 \} & \text{if}~~ v(0)_{j} = 0.
	\end{cases}
\end{equation*}
Here, the interesting case is $v(0)_j = 0$. In this case $w(i)_j = v(i)_j$, which shows that the maximum above is attained for $i = s(\vec v)_j$. Hence
\begin{equation*}
	v(0)_{j} + \sum_{i=1}^d \lambda_i w(i)_{j} \asymp_{d, \vec v(1), \dots, \vec v(d)} \lambda_{s(\vec v)_j}.
\end{equation*}
Our condition is then equivalent to the convergence of
\begin{equation*}
	\int_{\substack{(1 - 1/(d+1)) \Delta^d\\ \lambda_1\leq \lambda_2\leq \ldots \leq \lambda_d}} \prod_{j \text{ s.t.\ }  v(0)_{j} = 0} \lambda_{s(\vec v)_j}^{\ell_j} \,\mathrm d\vec \lambda = \int_{\substack{(1 - 1/(d+1)) \Delta^d\\ \lambda_1\leq \lambda_2\leq \ldots \leq \lambda_d}} \prod_{i=1}^d \lambda_i^{\sum_{j: v(0)_j=0, s(\vec v)_j=i}\ell_j} \, \mathrm d\vec \lambda .
\end{equation*}
Note that the integral over $\lambda_1$ is convergent if and only if
\begin{equation*}
\sum_{j: v(0)_j=0, s(\vec v)_j=1} \ell_j >-1,
\end{equation*}
and then the resulting integral over $\lambda_2$ is convergent if and only if
\begin{equation*}
1+ \sum_{j: v(0)_j=0, s(\vec v)_j=1} \ell_j +\sum_{j: v(0)_j=0, s(\vec v)_j=2} \ell_j >-1.
\end{equation*}
One now easily checks that the above integral converges if and only if for every $1\leq i\leq d$ we have
\begin{equation*}
i +  \sum_{\substack{1\le j\le s \\ v(0)_j=0,\ s(\vec v)_j\leq i}} \ell_j >0.
\end{equation*}
Since $v(0)_j \neq 0$ implies that $s(\vec v)_j = \infty$, the condition $v(0)_j = 0$ may be dropped in all these summations.
\end{proof}

\begin{remark}\label{rem:convergent_type_open}
By \cref{prop:convergence-test}, we observe that for a fixed face  $\cal F$, the condition of being of convergent type is an open condition for the vector $\vec\ell$ in the Euclidean topology on $\mathbb{R}^s$.
\end{remark}

Let $a = \max_{\vec x \in \cal P} |\vec x|_1$. We now state our main theorem in the case that $X_i=1$ for all $1\leq i\leq k$.

\begin{theorem}\label{maintheorem}
Let $f \colon \mathbb N^s \to \mathbb C$ be a function that satisfies \cref{mainassumption} and that is furthermore real and non-negative or satisfies \cref{upperboundassp}. Assume that the polytope $\cal P$ is bounded, that $\cal F$ is of dimension $d$ and not contained in a coordinate hyperplane, and that $\vec\ell$ is of convergent type for the face $\cal F$.

Then there exist $\varepsilon > 0$ and a constant $c_{\cal P}$, which only depends on $\vec \ell$ and the polytope $\cal P$, such that
    \begin{align*}
\Upsilon(B; \vec 1)
		= \frac{c_f c_{\cal P}}
		{\prod_{i=1}^s \alpha_i^{\ell_i}} B^a (\log B)^{d+\ell}
			+ O_{\vec \varpi, \vec \alpha, \vec \ell}(C_f B^a (\log B)^{d + \ell - \varepsilon}).
\end{align*}
\end{theorem}
We shall obtain an explicit description of the constant $c_{\cal P}$ in \cref{maintheorem_with_simplices}.

\subsection{Reduction to an integral}

We claim that for many combinations of $\vec \varpi$, $\vec \alpha$, and $\vec \ell$ the behaviour of the function $\Upsilon$ is fairly straightforward to predict.
In particular, one is able to explicitly express the asymptotic as an integral of some function over a domain of the above form.

In order to achieve this, we start by decomposing the summation of $\Upsilon$ into many small boxes. We pick a parameter $1<\theta \leq 2$ to govern the size of these boxes. As $\theta$ approaches $1$, the boxes become smaller and smaller. Consequently, we may hope that boxes that only partially intersect the summation domain of $\Upsilon$ will contribute a negligible amount to the overall sum.

The boxes of concern are defined as follows. For $\vec j \in \mathbb N^s$, we define
\begin{equation*}
	\cal B_{\vec j} = \prod_{i = 1}^s \begin{cases}
		[1, \theta^{1/\alpha_i}]                             & \text{if}~~ j_i = 1, \\
		(\theta^{(j_i - 1)/\alpha_i}, \theta^{j_i/\alpha_i}] & \text{if}~~ j_i > 1.
	\end{cases}
\end{equation*}
The definition ensures that the union of all boxes $\cal B_{\vec j}$ for $\vec j \in \mathbb N^s$ covers the entire domain $\mathbb R_{\geq 1}^s$. The following lemma describes which boxes lie entirely inside the summation of $\Upsilon$ and which boxes require more careful analysis (as they might partially intersect the summation domain).

\begin{lemma}\label{lem:decomposition into boxes}
	Let $\gamma > \max_{1\leq l\leq k} \sum_{i=1}^s \frac{\varpi_{li}}{\alpha_i}$ and set
	\begin{equation*}\textstyle
		\vec U_{\!-} = \bigl(\frac{\log (B/X_1)}{\log \theta}, \dots, \frac{\log (B/X_k)}{\log \theta}\bigr)
		\quad\text{and}\quad
		\vec U_{\!+} = \vec U_{\!-} + (\gamma, \dots, \gamma).
	\end{equation*}
	Then
	\begin{equation*}
		\Upsilon(B; \vec X)
		= \sum_{\vec j \in \cal P(\vec U_{\!-})} \sum_{\vec x \in \cal B_{\vec j}} f(\vec x)
		+ O\biggl(~ \sum_{\substack{\vec j \in \cal P(\vec U_{\!+}) \\ \vec j \notin \cal P(\vec U_{\!-})}} \sum_{\vec x \in \cal B_{\vec j}} |f(\vec x)| ~\biggr).
	\end{equation*}
	Here, the summations for $\vec j$ run through elements of $\mathbb N^s$.
\end{lemma}

\begin{proof}
	The proof follows from the following two observations.
	\begin{enumerate}[leftmargin=*]
		\item If $\vec j \in \cal P(\vec U_{\!-})$, then all $\vec x \in \cal B_{\vec j} \cap \mathbb N^s$ appear in the sum of $\Upsilon(B; \vec X)$. Indeed, if we consider a point $\vec x \in \cal B_{\vec j} \cap \mathbb N^s$ for such a value of $\vec j$, then $x_i\leq \theta^{j_i/\alpha_i}$, $1\leq i\leq s$ and for every $1\leq l\leq k$ we have
		      $$
			      \prod_{i=1}^s x_i^{\varpi_{li}} \leq \prod_{i=1}^s \theta^{\frac{\varpi_{li}j_i}{\alpha_i}} = \exp \left( \sum_{i=1}^s \frac{\varpi_{li}j_i}{\alpha_i} \log \theta \right) \leq \frac{B}{X_l}.
		      $$
		\item If $\vec j \notin \cal P(\vec U_{\!+})$, then no $\vec x \in \cal B_{\vec j} \cap \mathbb N^s$ appears in the sum of $\Upsilon(B; \vec X)$. Indeed, in this situation there exists a $1\leq l\leq k$ with
		      $$
			      \sum_{i=1}^s \frac{\varpi_{li}}{\alpha_i} j_i>\frac{\log (B/X_l)}{\log \theta} +\gamma,
		      $$
		      and hence
		      $$
			      \sum_{i=1}^s \frac{\varpi_{li}}{\alpha_i} (j_i-1)> \frac{\log (B/X_l)}{\log \theta} +\gamma - \sum_{i=1}^s \frac{\varpi_{li}}{\alpha_i} > \frac{\log (B/X_l)}{\log \theta} .       $$
		      From this we obtain $\prod_{1\leq i\leq s} x_i^{\varpi_{li}}> B/X_l$ and $\vec x$ does not appear in the sum of $\Upsilon(B; \vec X)$.
	\end{enumerate}
	This leaves the boxes $\cal B_{\vec j}$ for $\vec j \in \cal P(\vec U_{\!+}) \setminus \cal P(\vec U_{\!-})$ as the only boxes that might partially intersect the summation domain of $\Upsilon(B; \vec X)$, yielding the desired decomposition.
\end{proof}

Up to this point, we have not used any of the assumptions on the function $f$. The next step is to make use of the available asymptotic formula for $f$ to obtain an explicit expression for the first term in the previous lemma.

\begin{lemma}\label{lem:fromboxestointegrals}
	Let $\vec W = (\log (B/X_1), \dots, \log (B/X_k)) = (\log \theta) \vec U_{\!-}$ and $1< \theta \leq 2$. Then for $\gamma > \max_{1\leq l\leq k} \sum_{i=1}^s \frac{\varpi_{li}}{\alpha_i}$ we have
	\begin{multline*}
		\sum_{\vec j \in \cal P(\vec U_{\!-})} \sum_{\vec x \in \cal B_{\vec j}} f(\vec x)
		= c_f \Bigl(\prod_{i=1}^s \alpha_i^{-\ell_i}\Bigr) \int_{\cal P(\vec W)}  \prod_{i=1}^s \mathrm e^{t_i} (1+t_i )^{\ell_i} \,\mathrm d\vec t\\
		+ O\biggl(
		C_f (\log \theta)^{-s} \int_{\cal P(\vec W)} \bigl(1 + \min\nolimits_i t_i\bigr)^{-\delta} \prod_{i=1}^s \mathrm e^{t_i} (1 + t_i)^{\ell_i} \,\mathrm d\vec t
		\biggr)\\
		+ O\biggl(
		C_f \int_{\Xi'} \prod_{i=1}^s \mathrm e^{t_i} (1 + t_i)^{\ell_i} \,\mathrm d\vec t
		\biggr),
	\end{multline*}
	where
	\begin{equation*}
		\Xi' = \cal P(\vec W) \setminus \cal P(\vec W - (\gamma \log \theta, \dots, \gamma \log \theta)).
	\end{equation*}
\end{lemma}

A neat consequence of the above lemma is that the main term no longer depends on the parameter $\theta$. The second of the two error terms will be small since the volume of $\Xi'$ gets smaller as $\theta$ approaches $1$. In practice, once the size of the main term has been determined, this will be a straightforward consequence of Hölder's inequality. We also encounter our first issue that will constrain the setting in which this approach is viable: for some choices of parameters $\vec \varpi$, $\vec \alpha$, and $\vec \ell$, one will obtain
\begin{equation*}
	\int_{\cal P(\vec W)} \bigl(1 + \min\nolimits_i t_i\bigr)^{-\delta} \prod_{i=1}^s \mathrm e^{t_i} (1 + t_i)^{\ell_i} \,\mathrm d\vec t
	\asymp \int_{\cal P(\vec W)} \prod_{i=1}^s \mathrm e^{t_i} (1 + t_i)^{\ell_i} \,\mathrm d\vec t,
\end{equation*}
which makes it impossible to give sufficiently strong bounds for the error term. The simplest example of such a situation is the case $k = 2$, $\vec \varpi = \vec \alpha$ and $\ell_i \in (-\infty, -1)$ for some $i$.

\begin{proof}
	We begin by computing the contribution of a single box $\cal B_{\vec j}$.
	Using $1 < \theta \ll 1$ and the derivative 
	\begin{equation*}
		\Bigl( \theta^t (1 + t \log \theta)^\nu \Bigr)' = (\log \theta) \theta^t \Bigl( (1 + t \log \theta)^{\nu} + \nu (1 + t \log \theta)^{\nu - 1} \Bigr)
	\end{equation*}
	in $t$, we obtain the asymptotic
	\begin{multline*}
		\theta^j (1+ j \log \theta )^\nu - \theta^{j-1}(1+ (j-1) \log \theta )^\nu \\
		= (\log \theta) \int_{j-1}^j \theta^t (1 + t \log \theta)^\nu \,\mathrm dt + O_\nu\Bigl( \theta^j (1 + j \log \theta)^{\nu - 1} \Bigr)
	\end{multline*}
	 for $j\geq 1$.
	Inclusion-exclusion immediately yields
	\begin{multline*}
		\sum_{\vec r \in \{0, 1\}^s} (-1)^{|\vec r|_1} \prod_{i=1}^s \theta^{j_i - r_i} \Bigl( 1 + (j_i - r_i) \log \theta \Bigr)^{\ell_i}\\
		= (\log \theta)^s \int_{j_1 - 1}^{j_1} \cdots \int_{j_s - 1}^{j_s} \prod_{i=1}^s \theta^{t_i} (1 + t_i \log \theta)^{\ell_i} \,\mathrm d\vec t\\
		+ O\biggl( (1 + \min\nolimits_i j_i \log \theta)^{-\delta} \prod_{i=1}^s \theta^{j_i} (1 + j_i \log \theta)^{\ell_i} \biggr).
	\end{multline*}
	Now apply the asymptotic formula for $f$ with inclusion-exclusion to obtain
	\begin{multline*}
		\sum_{\vec x \in \cal B_{\vec j}} f(\vec x)
		= \sum_{\vec r \in \{0, 1\}^s} (-1)^{|\vec r|_1} \sum_{x_i \leq \theta^{(j_i - r_i)/\alpha_i}} f(\vec x)\\
		= c_f \sum_{\vec r \in \{0, 1\}^s} (-1)^{|\vec r|_1} \prod_{i=1}^s \theta^{j_i - r_i} \Bigl( 1 + \frac{j_i - r_i}{\alpha_i} \log \theta \Bigr)^{\ell_i}\\
		+ O\biggl(
		C_f \Bigl(1 + \min\nolimits_i \frac{j_i}{\alpha_i} \log \theta\Bigr)^{-\delta} \prod_{i=1}^s \theta^{j_i} \Bigl( 1 + \frac{j_i}{\alpha_i} \log \theta \Bigr)^{\ell_i}
		\biggr).
	\end{multline*}
	Since for $t \geq 0$ and $c > 0$, one has $(1 + c t)^\nu = c^\nu (1 + t)^\nu + O_{\nu,c}((1 + t)^{\nu - 1})$, we obtain
	\begin{multline*}
		\sum_{\vec x \in \cal B_{\vec j}} f(\vec x)
		= c_f \Bigl( \prod_{i=1}^s \alpha_i^{-\ell_i} \Bigr) (\log \theta)^s \int_{j_1 - 1}^{j_1} \cdots \int_{j_s - 1}^{j_s} \prod_{i=1}^s \theta^{t_i} (1 + t_i \log \theta)^{\ell_i} \,\mathrm d\vec t\\
		+ O\biggl(
		C_f \Bigl(1 + \min\nolimits_i j_i \log \theta\Bigr)^{-\delta} \prod_{i=1}^s \theta^{j_i} (1 + j_i \log \theta)^{\ell_i}
		\biggr).
	\end{multline*}
	
	We can artificially introduce an integration into the error term by bounding
	\begin{equation}\label{eq:artificial integration}
		\theta^{j_i} (1 + j_i \log \theta)^{\ell_i} \ll \int_{j_i - 1}^{j_i} \theta^{t_i} (1 + t_i \log \theta)^{\ell_i} \,\mathrm dt_i,
	\end{equation}
	which yields the final asymptotic for the contribution of a single box $\cal B_{\vec j}$ given by
	\begin{multline*}
		\sum_{\vec x \in \cal B_{\vec j}} f(\vec x)
		= c_f \Bigl( \prod_{i=1}^s \alpha_i^{-\ell_i} \Bigr) (\log \theta)^s \int_{j_1 - 1}^{j_1} \cdots \int_{j_s - 1}^{j_s} \prod_{i=1}^s \theta^{t_i} (1 + t_i \log \theta)^{\ell_i} \,\mathrm d\vec t\\
		+ O\biggl(
		C_f \int_{j_1 - 1}^{j_1} \cdots \int_{j_s - 1}^{j_s} (1 + \min\nolimits_i t_i \log \theta)^{-\delta} \prod_{i=1}^s \theta^{t_i} (1 + t_i \log \theta)^{\ell_i} \,\mathrm d\vec t
		\biggr).
	\end{multline*}
	We now take the sum over $\vec j \in \cal P(\vec U_{\!-})$. As in \cref{lem:decomposition into boxes}, we let
	\begin{equation*}
		\gamma > \max_{1\leq l\leq k} \sum_{i=1}^s \frac{\varpi_{li}}{\alpha_i}.
	\end{equation*}
	Note that if $t_i\leq j_i$, $1\leq i\leq s$, and $\vec j \in \cal P(\vec U_{\!-})$, then $\vec t\in \cal P(\vec U_{\!-})$. On the other hand, if $\vec t\in \cal P(\vec U_{\!-}- (\gamma, \ldots, \gamma))$, then $\vec j = (\lceil t_1\rceil, \ldots, \lceil t_s \rceil) \in \cal P(\vec U_{\!-})$ by the same calculation as in \cref{lem:decomposition into boxes}.
	We deduce that
	\begin{equation*}
		\cal P(\vec U_{-}) \setminus  \bigcup_{\vec j \in \cal P(\vec U_{\!-})} \times_{i=1}^s [j_i-1,j_i] \subset	\cal P(\vec U_{-}) \setminus \cal P(\vec U_{-} - (\gamma, \ldots, \gamma)).
	\end{equation*}

	The set on the right-hand side is equal to $(\log \theta)^{-1} \Xi'$. Consequently, we obtain
	\begin{multline*}
		\sum_{\vec j \in \cal P(\vec U_{\!-})} \sum_{\vec x \in \cal B_{\vec j}} f(\vec x)
		= c_f \Bigl( \prod_{i=1}^s \alpha_i^{-\ell_i} \Bigr) (\log \theta)^s \int_{\cal P(\vec U_{-})} \prod_{i=1}^s \theta^{t_i} (1 + t_i \log \theta)^{\ell_i} \,\mathrm d\vec t\\
		+ O\biggl(
		C_f \int_{\cal P(\vec U_{-})} (1 + \min\nolimits_i t_i \log \theta)^{-\delta} \prod_{i=1}^s \theta^{t_i} (1 + t_i \log \theta)^{\ell_i} \,\mathrm d\vec t
		\biggr)\\
		+ O\biggl(
		C_f (\log \theta)^s \int_{(\log \theta)^{-1}\Xi'} \prod_{i=1}^s \theta^{t_i} (1 + t_i \log \theta)^{\ell_i} \,\mathrm d\vec t
		\biggr),
	\end{multline*}
	where the second of the two error terms arises from our completion of the main term to the entire domain $\cal P(\vec U_{-})$. Substituting $t_i = (\log \theta)^{-1} t_i'$ in all integrals and using $(\log \theta) \cal P(\vec U_{-}) = \cal P(\vec W)$, we obtain the asymptotic claimed in the lemma.
\end{proof}

The same approach can now be applied to the error term of \cref{lem:decomposition into boxes} to obtain an explicit expression for the second term.

\begin{lemma}\label{lem:boundaryboxes}
	Let $\vec W = (\log (B/X_1), \dots, \log (B/X_k))$ as in the previous lemma, $\vec U_{\!\pm}$ be defined as in \cref{lem:decomposition into boxes}, and $1< \theta \leq 2$. Let $\gamma > \max_{1\leq l\leq k} \sum_{i=1}^s \frac{\varpi_{li}}{\alpha_i}$. Then
	\begin{multline*}
		\sum_{\substack{\vec j \in \cal P(\vec U_{\!+}) \\ \vec j \notin \cal P(\vec U_{\!-})}} \sum_{\vec x \in \cal B_{\vec j}} |f(\vec x)| \ll
		C_f \int_{\Xi} \prod_{i=1}^s \mathrm e^{t_i} (1 + t_i)^{\ell_i} \,\mathrm d\vec t\\
		+ C_f (\log \theta)^{-A - s} \int_{\Xi} (1 + \min\nolimits_i t_i)^{-\delta} \prod_{i=1}^s \mathrm e^{t_i} (1 + t_i)^{\ell_i} \,\mathrm d\vec t,
	\end{multline*}
	where
	\begin{equation*}
		\Xi = \cal P(\vec W + (\gamma \log \theta, \dots, \gamma \log \theta)) \setminus \cal P(\vec W - (\gamma \log \theta, \dots, \gamma \log \theta)).
	\end{equation*}
\end{lemma}

\begin{proof}
	By definition of $\cal B_{\vec j}$ for $\kappa = \max_i \alpha_i^{-1}$, we have
	\begin{equation*}
		\cal B_{\vec j} \subseteq \Bigl\{\,\vec x \in \mathbb R_{\geq 1}^s ~:~ \theta^{(j_i - 1)/\alpha_i} \leq x_i \leq \theta^{(j_i-1)/\alpha_i} \theta^{\kappa} ~\forall i \,\Bigr\}.
	\end{equation*}
	Hence, we can apply \cref{upperboundassp} to obtain
	\begin{equation*}
		\sum_{\vec x \in \cal B_{\vec j}} |f(\vec x)|
		\ll C_f \Bigl( (\log \theta)^{s} + (\log \theta)^{-A} (1 + \min\nolimits_i j_i \log \theta)^{-\delta} \Bigr) \prod_{i=1}^s \theta^{j_i} (1 + j_i \log \theta)^{\ell_i}.
	\end{equation*}
	As we did in \cref{eq:artificial integration}, we can artificially introduce an integration, so that summing over $\vec j$ tiles a larger integral.
	Overestimating the domain of integration as we did in \cref{lem:fromboxestointegrals} and using the inequality $\gamma > \max_{1\leq l\leq k} \sum_{i=1}^s \frac{\varpi_{li}}{\alpha_i}$ yields
	\begin{multline*}
		\sum_{\substack{\vec j \in \cal P(\vec U_{\!+}) \\ \vec j \notin \cal P(\vec U_{\!-})}} \sum_{\vec x \in \cal B_{\vec j}} |f(\vec x)|
		\ll C_f (\log \theta)^{s} \int_{\Xi_0} \prod_{i=1}^s \theta^{t_i} (1 + t_i \log \theta)^{\ell_i} \,\mathrm d\vec t\\
		+ C_f (\log \theta)^{-A} \int_{\Xi_0} (1 + \min\nolimits_i t_i \log \theta)^{-\delta} \prod_{i=1}^s \theta^{t_i} (1 + t_i \log \theta)^{\ell_i} \,\mathrm d\vec t,
	\end{multline*}
	where
	\begin{equation*}
		\Xi_0 = \cal P(\vec U_{\!+}) \setminus \cal P(\vec U_{\!-} - (\gamma, \dots, \gamma)).
	\end{equation*}
	Substituting $t_i = (\log \theta)^{-1} t_i'$ as was done in \cref{lem:fromboxestointegrals} yields the desired asymptotic.
\end{proof}

Combined, the results of the previous two lemmas yield the following corollary.

\begin{corollary}\label{lem:sumtointegral}
	Let $\vec W = (\log (B/X_1), \ldots, \log (B/X_k))$ and $1< \theta \leq 2$. Then, for $\gamma > \max_{1\leq l\leq k} \sum_{i=1}^s \frac{\varpi_{li}}{\alpha_i}$, we have
	\begin{multline*}
		\Upsilon(B; \vec X)
		= c_f \Bigl( \prod_{i=1}^s \alpha_i^{-\ell_i} \Bigr) \int_{\cal P(\vec W)}  \prod_{i=1}^s \mathrm e^{t_i} (t_i + 1)^{\ell_i} \,\mathrm d\vec t\\
		+ O\biggl(
		C_f (\log \theta)^{-A - s} \int_{\cal P(\vec W + (\gamma \log \theta,\ldots,\gamma\log\theta))} (\min\nolimits_i t_i + 1)^{-\delta} \prod_{i=1}^s \mathrm e^{t_i} (t_i + 1)^{\ell_i} \,\mathrm d\vec t
		\biggr)\\
		+ O\biggl(
		C_f \int_{\Xi} \prod_{i=1}^s \mathrm e^{t_i} (t_i + 1)^{\ell_i} \,\mathrm d\vec t
		\biggr),
	\end{multline*}
	where
	\begin{equation*}
		\Xi = \cal P(\vec W + (\gamma \log \theta, \dots, \gamma \log \theta)) \setminus \cal P(\vec W - (\gamma \log \theta, \dots, \gamma \log \theta)).
	\end{equation*}
\end{corollary}

With this corollary at hand, we reduced obtaining an asymptotic formula for $\Upsilon(B; \vec X)$ to finding one for the integral
\begin{equation}\label{eq:main integral}
	\int_{\cal P(\vec W)}  \prod_{i=1}^s \mathrm e^{t_i} (t_i + 1)^{\ell_i} \,\mathrm d\vec t.
\end{equation}
As long as decreasing any of the $\ell_i$ leads to a smaller order of growth, one will be able to give sufficiently strong bounds for the error terms.

\subsection{Upper bounds}

Before returning to estimates for integrals of the form \cref{eq:main integral} and similar, we provide a simple but very useful lemma.

\begin{lemma}\label{lem:VolumeGrowthEstimate}
	Let $A_0 \subset A_1 \subset A_2 = \mathbb R^s$ be affine subspaces of dimension $d_0 < d_1$ and $d_2 = s$ equipped with the Euclidean distance. Let $X_i \subset A_i$ be compact subsets with $\operatorname{diam}(X_0) \leq \delta_0$ and $\operatorname{dist}(X_{i-1}, X_i) \leq \delta_i$ for $i = 1, 2$. Then
	\begin{equation*}
		\operatorname{vol}(X_2) \ll_{d_0, d_1, d_2} (\delta_0 + \delta_1 + \delta_2)^{d_0} (\delta_1 + \delta_2)^{d_1 - d_0} \delta_2^{d_2 - d_1}. 
	\end{equation*}
	When $\delta_0 \geq \delta_1 \geq \delta_2$, this simplifies to $\operatorname{vol}(X_2) \ll_{d_0, d_1, d_2} \delta_0^{d_0} \delta_1^{d_1 - d_0} \delta_2^{d_2 - d_1}$.
\end{lemma}

The statement of the lemma gives a bound for a set that is first enlarged in a low-dimensional subspace, and then enlarged in the full-dimensional space. It may be interpreted as a weak inequality version of Minkowski--Steiner formula applied twice. Using this formula would yield much more general results, but would require the introduction of intrinsic volumes. To avoid this, we give a direct and straightforward proof.

\begin{proof}
	Volumes and distances are invariant under translations and orthogonal transformations. Hence, we may choose coordinates such that the spaces $A_i$ are given by 
    \begin{align*}
		A_i = \mathbb R^{d_i} \times \{0\}^{s-d_i}.
    \end{align*}
	We bound each $X_i$ iteratively by Cartesian products of balls. Using a translation, we may assume that the origin is contained in $X_0$.
	
    First, since $X_0 \subset A_0$ and has diameter at most $\delta_0$, it is entirely contained in a $d_0$-dimensional ball of radius $\delta_0$. Fixing our origin at the centre of this bounding ball, we have
	\begin{equation*}
		X_0 \subset [-\delta_0, \delta_0]^{d_0} \times \{0\}^{s-d_0}.
	\end{equation*}
	Now $X_1 \subset A_1$ and $\operatorname{dist}(X_0, X_1) \leq \delta_1$. Since $X_0$ and $X_1$ are contained in $A_1$, we obtain
    \begin{equation*}
		X_1 \subset [-\delta_0 - \delta_1, \delta_0 + \delta_1]^{d_0} \times [-\delta_1, \delta_1]^{d_1-d_0} \times \{0\}^{s-d_1}.
    \end{equation*} 
	The same argument shows that by $\operatorname{dist}(X_1, X_2) \leq \delta_2$, we have
	\begin{equation*}
		X_2 \subset [-\delta_0 - \delta_1 - \delta_2, \delta_0 + \delta_1 + \delta_2]^{d_0} \times [-\delta_1 - \delta_2, \delta_1 + \delta_2]^{d_1-d_0} \times [-\delta_2, \delta_2]^{s-d_1}.
	\end{equation*}
    The claimed bound for the volume of $X_2$ now follows.
\end{proof}

We next provide an upper bound for the integral in \cref{eq:main integral}. For this step, we only consider the polytope
$\cal P(\log B)$.
This significantly simplifies things, as the general shape of our polytope does not depend on $B$ by~\eqref{eq:polytope-scaling}. A useful consequence is that any vertex of $\cal P(\log B)$ that is not a vertex of the maximal face $\cal F(\log B)$ is far away from the maximal face. Recall that $a = \max_{\vec x \in \cal P} |\vec x|_1$.

\begin{lemma}\label{lem:UpperBoundNoX}
	For all sufficiently large $B$ we have
	\begin{equation*}
		\int_{\cal P(\log B)} \prod_{i=1}^s \mathrm e^{t_i} (t_i + 1)^{\ell_i} \,\mathrm d\vec t
		\ll_{A'} B^a (\log B)^{-A'} 
        +  B^a \int_{\cal F(\log B)} \prod_{i=1}^s  (t_i + 1)^{\ell_i} \,\mathrm d \mu,
    \end{equation*}
	where $A' > 0$ may be chosen arbitrarily large.
    Here, $\mathrm d\mu$ is the $d$-dimensional measure on $\cal F(\log B)$ that is induced by the standard scalar product, where $d = \dim (\cal F)$. If $d=0$, then we understand the integral as
    \begin{equation*}
\int_{\cal F(\log B)} \prod_{i=1}^s (t_i + 1)^{\ell_i} \,\mathrm d \mu =\sum_{\vec t\in \cal F(\log B)} \prod_{i=1}^s (t_i + 1)^{\ell_i},
    \end{equation*} 
	which is well-defined since $d = 0$ implies that $\cal F(\log B)$ contains only a single point.
    \end{lemma}

\begin{proof}
	We denote $a = \max_{\vec t \in \cal P} |\vec t|_1$. Since $\cal P$ has only a finite number of vertices, there exists a constant $\delta > 0$ such that no vertex $\vec v$ of $\cal P$ satisfies $|\vec v|_1 \in (a - \delta, a)$. We then cover $\cal P$ by the two sets
	\begin{equation*}
		\cal P^{-} = \{\vec t \in \cal P : |\vec t|_1 \leq a - \delta\}, \qquad\text{and}\qquad \cal P^{+} = \{\vec t \in \cal P : a - \delta \leq |\vec t|_1 \leq a\}.
	\end{equation*}
	Trivially estimating the integral over $\cal P^{-}$, we obtain
	\begin{equation}\label{eq:BoundPminus}
	\begin{aligned}
		\int_{(\log B) \cal P^{-}} \prod_{i=1}^s \mathrm e^{t_i} (t_i + 1)^{\ell_i} \,\mathrm d\vec t
		&\leq B^{a - \delta} (a(\log B) + 1)^{|\vec \ell|_1} \vol\bigl((\log B) \cal P^{-}\bigr)\\
		&\ll_{\delta} B^{a - \delta/2}.
	\end{aligned}
	\end{equation}

	By assumption on the choice of $\delta$, the set $\cal P^{+}$ is a polytope for which all vertices $\vec v$ satisfy $|\vec v|_1 \in \{a - \delta, a\}$. Hence, we can cover $\cal P^{+}$ by a finite number of simplices for which all $s + 1$ vertices satisfy the same property. Thus, it is sufficient to prove the lemma for every simplex $\cal S$ with vertices $\vec v_0, \ldots, \vec v_s$ satisfying $|\vec v_i|_1 \in \{a - \delta, a\}$ for all $0 \leq i \leq s$. If the maximal face $\cal F$ has dimension $d$, then at most $d + 1$ vertices of $\cal S$ can be contained in $\cal F$ (otherwise, the convex hull of these vertices has no volume). We split the proof into the two cases for which exactly $d + 1$ vertices of $\cal S$ are contained in $\cal F$ and for which at most $d$ vertices of $\cal S$ are contained in $\cal F$.

	\smallskip

	Assume that exactly $d + 1$ vertices of $\cal S$ are contained in $\cal F$. Without loss of generality, we may assume that these vertices are $\vec v_0, \ldots, \vec v_d$. Any point $\vec t \in \cal S$ can be written as a convex combination of the vertices of $\cal S$, that is,
	\begin{equation}\label{eq:ConvexCombination}
		\vec t = \sum_{j = 0}^s u_j \vec v_j, \qquad\text{with}\qquad u_j \geq 0 \quad\text{and}\quad \sum_{j = 0}^s u_j = 1.
	\end{equation}
	Given such a representation, we have
	\begin{equation*}
		\sum_{i = 1}^s t_i = \sum_{j = 0}^s \sum_{i=1}^s u_j v_{j,i} = a - \sum_{j = d+1}^s u_j \delta.
	\end{equation*}
	Since the points $\vec u = (u_1, \ldots, u_s) \in \Delta^s$ are in bijection with the points $\vec t \in \cal S$ via \cref{eq:ConvexCombination}, we can perform a linear change of variables to obtain the upper bound
	\begin{multline}\label{eq:ConvexCombinationBound}
		\int_{(\log B) \cal S} \prod_{i=1}^s \mathrm e^{t_i} (t_i + 1)^{\ell_i} \,\mathrm d\vec t\\
		\ll_{\cal S} B^a \int_{(\log B) \Delta^s} \Bigl(\vec 1 + (\log B) \vec v_0 + \sum_{j = 1}^s u_j (\vec v_j - \vec v_0) \Bigr)^{\vec \ell} \exp\biggl(- \sum_{j = d+1}^s u_j \delta \biggr) \,\mathrm d\vec u.
	\end{multline}
	Here, we can apply the crude upper bound
	\begin{multline*}
		\Bigl(\vec 1 + (\log B) \vec v_0 + \sum_{j = 1}^s u_j (\vec v_j - \vec v_0) \Bigr)^{\vec \ell}\\
		\ll_{\cal S} \Bigl(\vec 1 + (\log B) \vec v_0 + \sum_{j = 1}^d u_j (\vec v_j - \vec v_0) \Bigr)^{\vec \ell} \Bigl(1 + \sum_{j = d+1}^s u_j\Bigr)^{|\vec \ell|_1}.
	\end{multline*}
	This inequality is easily derived from the fact that given any $x, y \in \mathbb R$ with $x \geq 0$ and $x + y \geq 0$, we have
	\begin{equation}\label{eq:CrudeBound}
		(x + y + 1)^{\nu} \ll_\nu (x + 1)^{\nu} (|y| + 1)^{|\nu|}
	\end{equation}
	by case distinctions $\nu \gtrless 0$ and $x \gtrless 2|y|$. By additionally overestimating the integration over $\vec u \in (\log B) \Delta^s \subseteq (\log B) \Delta^d \times (\log B) \Delta^{s-d}$, we obtain
	\begin{align*}
		\int_{(\log B) \cal S} \prod_{i=1}^s \mathrm e^{t_i} (t_i + 1)^{\ell_i} \,\mathrm d\vec t
		\ll_{\cal S} B^a \int_{(\log B) \Delta^d} \Bigl(\vec 1 + (\log B) \vec v_0 + \sum_{j = 1}^d u_j (\vec v_j - \vec v_0) \Bigr)^{\vec \ell} \,\mathrm d\vec u\\
		\times \int_{(\log B) \Delta^{s-d}} (1 + |\vec u'|_1)^{|\vec \ell|_1} \exp\biggl(- |\vec u'|_1 \delta \biggr) \,\mathrm d\vec u'.
	\end{align*}
	The second integral is absolutely convergent and increasing with $B$ and may be estimated by a constant. Since the vertices $\vec v_0, \ldots, \vec v_d$ are contained in the maximal face, reversing the linear change of variables $\vec t = \vec v_0 + \sum_{j = 1}^d u_j (\vec v_j - \vec v_0)$ bijects $(\log B) \Delta^d$ onto $(\log B) (\cal S \cap \cal F)$. Hence, we obtain
	\begin{equation}\label{eq:FullDimBound}
		\int_{(\log B) \cal S} \prod_{i=1}^s \mathrm e^{t_i} (t_i + 1)^{\ell_i} \,\mathrm d\vec t
		\ll_{\cal S} B^a \int_{(\log B) \cal F} \prod_{i=1}^s (t_i + 1)^{\ell_i} \,\mathrm d\mu,
	\end{equation}
	which proves the desired bound for these simplices.

	\smallskip

	Finally, we are left to consider the cases for which $d' \leq d$ vertices of $\cal S$ are contained in $\cal F$. These pieces are expected to contribute less to the integral, as they have smaller volume close to the maximal face. To obtain an upper bound, we artificially `drag' $d-d'$ vertices of $\cal S$ onto the maximal face and obtain an upper bound by the previous case.
	Start by assuming that $\vec v_0, \ldots, \vec v_{d'}$ are contained in $\cal F$.
	We construct a new simplex $\cal S'$ with vertices $\vec w_0, \ldots, \vec w_s$ as follows. Let $\vec w_i = \vec v_i$, for $0\le i\le d'$, be the vertices of $\cal S$ that are contained in $\cal F$.
	Choose additional points $\vec w_{d'+1}, \ldots, \vec w_d$ contained in $\cal F$ such that the convex hull of $\vec w_0, \ldots, \vec w_d$ is a $d$-dimensional simplex. Since $\cal S$ is an $s$-dimensional simplex, we can choose the remaining vertices $\vec w_{d+1}, \ldots, \vec w_s$ among $\{ \vec v_{d'+1}, \ldots, \vec v_s \}$
	so that the convex hull of $\vec w_0, \ldots, \vec w_s$ is an $s$-dimensional simplex. By reordering the indices of $\vec v_i$, we may assume that $\vec w_i = \vec v_i$ for $d + 1 \leq i \leq s$. This construction ensures that
	\begin{equation}\label{eq:VWEqualRange}
		\vec v_i = \vec w_i \quad\text{for all}\quad 0 \leq i \leq d' \quad\text{and all}\quad d + 1 \leq i \leq s.
	\end{equation}
	As in \eqref{eq:ConvexCombination}, write the points in $\cal S$ as convex combinations of their vertices. We obtain the bound from~\eqref{eq:ConvexCombinationBound} with $d$ replaced by $d'$. Let $A'$ be arbitrarily large but fixed. We cover the set $(\log B) \Delta^s$ by
	\begin{align*}
		\cal U^{-} &= \Bigl\{\vec u \in \Delta^s : \sum_{j = d'+1}^s u_j \geq A'(\log \log B)(\log B)^{-1} \Bigr\} \quad \text{and}\\
		\cal U^{+} &= \Bigl\{\vec u \in \Delta^s : \sum_{j = d'+1}^s u_j \leq A'(\log \log B)(\log B)^{-1} \Bigr\}.
	\end{align*}
	Looking at the integral from \eqref{eq:ConvexCombinationBound}, we can apply the trivial bounds on $(\log B) \cal U^{-}$ as in \eqref{eq:BoundPminus} to obtain
	\begin{multline*}
		\int_{(\log B) \cal S} \prod_{i=1}^s \mathrm e^{t_i} (t_i + 1)^{\ell_i} \,\mathrm d\vec t
		\ll_{\cal S, \delta, A'} B^a (\log B)^{-\delta A'} (\log B)^{|\vec \ell|_1 + s}\\
		+ \int_{(\log B) \cal U^{+}} \Bigl(\vec 1 + (\log B) \vec v_0 + \sum_{j = 1}^s u_j (\vec v_j - \vec v_0) \Bigr)^{\vec \ell} \exp\biggl(- \sum_{j = d'+1}^s u_j \delta \biggr) \,\mathrm d\vec u.
	\end{multline*}
	For the remaining integral over $(\log B)\cal U^{+}$, we attempt to replace $\vec v$ by $\vec w$ using brute force. Notice that both
	\begin{equation*}
		\vec t = (\log B) \vec v_0 + \sum_{j = 1}^s u_j (\vec v_j - \vec v_0) \quad\text{and}\quad \vec t' = (\log B) \vec w_0 + \sum_{j = 1}^s u_j (\vec w_j - \vec w_0)
	\end{equation*}
	are contained in the polytope $(\log B) \cal P$ and therefore only have non-negative coordinates. Let $\vec u \in (\log B) \cal U^{+}$. Then by \cref{eq:VWEqualRange} the above points $\vec t$ and $\vec t'$ satisfy
	\begin{equation*}
		|\vec t - \vec t'| = \Bigl| \sum_{j = d'+1}^d u_j (\vec v_j - \vec w_j) \Bigr| \ll_{\cal S, \cal S'} \sum_{j = d'+1}^d u_j.
	\end{equation*}
	We separate this contribution and the corresponding factors $\mathrm e^{-\delta u_j}$ from the remaining integrand. Using the crude estimate from \eqref{eq:CrudeBound}, we obtain
	\begin{multline*}
		\int_{(\log B) \cal S} \prod_{i=1}^s \mathrm e^{t_i} (t_i + 1)^{\ell_i} \,\mathrm d\vec t
		\ll_{\cal S, \cal S', \delta, A'} B^a (\log B)^{-\delta A' + |\vec \ell|_1 + s}\\
		+ \int_{(\log B) \cal U^{+}} \Bigl(\vec 1 + (\log B) \vec w_0 + \sum_{j = 1}^s u_j (\vec w_j - \vec w_0) \Bigr)^{\vec \ell} \exp\biggl(- \sum_{j = d+1}^s u_j \delta \biggr)\\
		\times \Bigl(1 + \sum_{j = d'+1}^d u_j\Bigr)^{|\vec \ell|_1} \exp\biggl(- \sum_{j = d'+1}^d u_j \delta \biggr) \,\mathrm d\vec u.
	\end{multline*}
	Notice that the terms on the last line have the form $(1 + x)^{|\vec \ell|_1} \mathrm e^{-\delta x}$, which for $x \geq 0$ is bounded by a constant depending on $\delta$. We may now overestimate the domain of integration $(\log B) \cal U^{+}$ by $(\log B) \Delta^s$. The affine change of variables $\vec t' = (\log B) \vec w_0 + \sum_{j = 1}^s u_j (\vec w_j - \vec w_0)$ now bijects $(\log B) \Delta^s$ onto $(\log B) \cal S'$. Hence, we obtain
	\begin{multline*}
		\int_{(\log B) \cal S} \prod_{i=1}^s \mathrm e^{t_i} (t_i + 1)^{\ell_i} \,\mathrm d\vec t \\
		\ll_{\cal S, \cal S', \delta, A'} B^a (\log B)^{-\delta A' + |\vec \ell|_1 + s}
		+ \int_{(\log B) \cal S'} \prod_{i=1}^s \mathrm e^{t_i} (t_i + 1)^{\ell_i} \,\mathrm d\vec t.
	\end{multline*}
	Since $\cal S'$ has $d + 1$ vertices contained in the maximal face, we can apply the bound from~\eqref{eq:FullDimBound} to obtain the lemma by choosing $A'$ sufficiently large.
\end{proof}

\begin{lemma}\label{lem:basicboundaryint}
Let $\vec\ell$ be of convergent type for the face $\cal F$ and $1\leq j_0\leq s$. Then there exists $\epsilon >0$ such that for $0<Y<1$ we have the bounds

\begin{align*}
\int_{\substack{\cal F\\ |t_{j_0}|\leq Y}} \prod_{i=1}^s t_i^{\ell_i} d\mu \ll_{\epsilon} Y^{\epsilon} \quad \mbox{ and }\quad \int_{\substack{\cal F\\ |t_{j_0}|\leq Y}} \prod_{i=1}^s \left(t_i+Y\right)^{\ell_i} d\mu \ll_{\epsilon} Y^{\epsilon}.
\end{align*}

\end{lemma}

\begin{proof}

For sufficiently small $\epsilon>0$, we estimate the first integral as 
\begin{align*}
\int_{\substack{\cal F\\ |t_{j_0}|\leq Y}} \prod_{i=1}^s t_i^{\ell_i} d\mu  \leq \int_{\cal F} \left(\frac{Y}{t_{j_0}}\right)^{\epsilon} \prod_{i=1}^s t_i^{\ell_i} d\mu \ll_\epsilon Y^{\epsilon}.
\end{align*}
Observe that the integral on the right-hand side is still convergent for $\epsilon$ sufficiently small, as being of convergent type is an open condition.\par
We start estimating the second integral by
\begin{align*}
\int_{\substack{\vec t \in \cal F\\ |t_{j_0}|\leq Y}} \prod_{i=1}^s \left(t_i+Y\right)^{\ell_i} d\mu  \ll_{\epsilon} \int_{\cal F} \left(\frac{Y}{t_{j_0}+Y}\right)^{\epsilon} \prod_{i=1}^s \left(t_i+Y\right)^{\ell_i} d\mu .
\end{align*}
As being of convergent type is an open condition, it is now sufficient to show that the integral 
\begin{align*}
 \int_{\cal F} \prod_{i=1}^s \left(t_i+Y\right)^{\ell_i} d\mu \ll 1
 \end{align*}
 is convergent if $\vec \ell$ is of convergent type. For this, let once more $\cal E$ be the set of vertices of $\cal F$ and $\vec v_0, \vec v_1, \dots, \vec v_{d} \in \cal E$ a choice of elements spanning a subsimplex of full dimension, and write $\vec w_i = \vec v_i - \vec v_0$. We need to show that for any such choice

\begin{equation*}
	 \int_{\substack{(1 - 1/(d+1)) \Delta^d\\\lambda_1\leq \ldots \leq \lambda_d}} \Bigl( \vec v_0 + \sum_{i=1}^d \lambda_i \vec w_i + Y\mathbf{1}\Bigr)^{\vec \ell} \,\mathrm d\vec \lambda \ll 1.
\end{equation*}
Hence, we need to bound the integral
\begin{align*}
J_{\vec \ell} \coloneq \int_{\substack{(1 - 1/(d+1)) \Delta^d\\\lambda_1\leq \ldots \leq \lambda_d}} \prod_{\substack{j=1\\ v_{0,j}=0}}^s \left(\lambda_{s_j}+Y\right)^{\ell_j} d\vec \lambda =  \int_{\substack{(1 - 1/d) \Delta^d\\\lambda_1\leq \ldots \leq \lambda_d}} \prod_{m=1}^d \left(\lambda_{i}+Y\right)^{\ell'_m} d\vec \lambda,
\end{align*}
where $\ell'_m = \sum_{s_j=m, v_{0,j}=0}\ell_j$.Observe that for $\mu>-1$ and $0<Y, Z \ll 1$ we have
\begin{align*}
\int_{\lambda\leq Z} (\lambda+Y)^{\mu} d\lambda \ll_\mu Y^{\mu+1}+Z^{\mu +1}\ll_\mu (Y+Z)^{\mu+1}.
\end{align*}
Inductively, we deduce that for $1\leq i\leq d $ we have
\begin{align*}
J_{\vec \ell}\ll \int_{\lambda_i\leq \ldots \leq \lambda_d} (\lambda_i+Y)^{(i-1)+\sum_{s_j\leq i-1, v_{0,j}=0}\ell_j}\prod_{m=i}^d (\lambda_m+Y)^{\ell'_m} d \lambda_i\ldots d\lambda_d
\end{align*}
which is convergent by the definition of being of convergent type.
\end{proof}

\begin{lemma}\label{lem:FaceIntegralAsymptotic}
Let $\vec\ell$ be of convergent type for the face $\cal F$ and assume that $\cal F$ is not contained in a coordinate hyperplane. Then there exists $\epsilon >0$ such that we have

\begin{align*}
\int_{\cal F(\log B)} \prod_{i=1}^s (t_i+1)^{\ell_i} d\mu = (\log B)^{d+\ell} \int_{\cal F} \prod_{i=1}^s t_i^{\ell_i} d\mu \left(1+O_\epsilon\!\left((\log B)^{-\epsilon}\right)\right).
\end{align*}

\end{lemma}

\begin{proof}
Recall that the property of being of convergent type is an open condition in $\vec \ell$. After scaling by a factor of $\log B$, we first compute, for $\epsilon>0$ sufficiently small, the integral
\begin{align*}
&\int_{\substack{\cal F\\ t_j\geq \frac{1}{\log B} \forall j}} \prod_{i=1}^s \left(t_i+\frac{1}{\log B}\right)^{\ell_i} d\mu = \int_{\substack{\cal F\\ t_j\geq \frac{1}{\log B} \forall j}} \prod_{i=1}^s \left(t_i^{\ell_i}+O\!\left(\ell_i\frac{1}{\log B}t_i^{\ell_i-1}\right)\right) d\mu\\
&= \int_{\substack{\cal F\\ t_j\geq \frac{1}{\log B} \forall j}} \prod_{i=1}^s \left(t_i^{\ell_i}+O\!\left(\left(\frac{1}{\log B}\right)^{\epsilon}t_i^{\ell_i-\epsilon}\right)\right) d\mu \\
&= \int_{\substack{\cal F\\ t_j\geq \frac{1}{\log B} \forall j}} \prod_{i=1}^st_i^{\ell_i} d\mu + O_\epsilon\!\left((\log B)^{-\epsilon}\right).
\end{align*}
The lemma now follows by applying \cref{lem:basicboundaryint} to the terms
\begin{align*}
\int_{\substack{\cal F\\ t_{j_0}\leq \frac{1}{\log B} }} \prod_{i=1}^st_i^{\ell_i} d\mu \quad \mbox{ and } \quad \int_{\substack{\cal F\\ t_{j_0}\leq \frac{1}{\log B} }} \prod_{i=1}^s \left(t_i+\frac{1}{\log B}\right)^{\ell_i} d\mu
\end{align*}
for some $1\leq j_0\leq s$.
\end{proof}

With the above lemma, we are able to give upper bounds for integrals over `small' subsets of the polytope $\cal P(\log B)$.

\begin{lemma}\label{lem:Hoelder}
	Let $\vec \ell$ be of convergent type for the face $\cal F$ and assume that $\cal F$ is not contained in a coordinate hyperplane. Let $\Xi \subset \cal P(\log B)$ be a measurable set. Then there exists a $\epsilon > 0$ such that for all $A'>0$, 
	\begin{equation*}
		\int_{\Xi} \prod_{i=1}^s \mathrm e^{t_i} (t_i + 1)^{\ell_i} \,\mathrm d\vec t
		\ll_{A'} \biggl(\Bigl(\frac{\vol(\Xi)}{(\log B)^d}\Bigr)^\epsilon + (\log B)^{-A'}\biggr) B^a (\log B)^{d + \ell}.
	\end{equation*}
\end{lemma}

Using this result, we may discard the contribution of any subset of $\cal P(\log B)$ with volume $o((\log B)^d)$. In practice, we want to recover savings of size $(\log B)^{-\epsilon}$ for some $\epsilon > 0$, so any subset of volume $O((\log B)^{d - \epsilon})$ will give negligible contributions.

\begin{proof}
	By \cref{rem:convergent_type_open}, being of convergent type for $\cal F$ is satisfied by an open set. Hence, for some sufficiently small $\epsilon > 0$, the vector $\vec \ell' = \frac{1}{1-\epsilon} \vec \ell$ is of convergent type for $\cal F$. We may now apply Hölder's inequality to obtain
	\begin{equation*}
		\int_{\Xi} \prod_{i=1}^s \mathrm e^{t_i} (t_i + 1)^{\ell_i} \,\mathrm d\vec t
		\leq \biggl( \int_{\Xi} \prod_{i=1}^s \mathrm e^{t_i} \,\mathrm d\vec t \biggr)^{\epsilon} \biggl( \int_{\cal P(\log B)} \prod_{i=1}^s \mathrm e^{t_i} (t_i + 1)^{\ell'_i} \,\mathrm d\vec t \biggr)^{1 - \epsilon}.
	\end{equation*}
	The product over $\mathrm e^{t_i}$ in the first integral is trivially bounded by $B^a$, which leaves only the volume of $\Xi$. Bounding the second integral using \cref{lem:UpperBoundNoX} yields
	\begin{multline}\label{eq:SmallVolume:1}
		\int_{\Xi} \prod_{i=1}^s \mathrm e^{t_i} (t_i + 1)^{\ell_i} \,\mathrm d\vec t
		\ll_{A'} \bigl(\vol(\Xi) B^a\bigr)^{\epsilon}\\
		\times \biggl(B^a (\log B)^{-A'} + B^a \int_{\cal F(\log B)} \prod_{i=1}^s (t_i + 1)^{\ell'_i} \,\mathrm d\mu \biggr)^{1 - \epsilon},
	\end{multline}
	where $A' > 0$ may be chosen arbitrarily large. Since $\vec \ell'$ is of convergent type, we may now apply \cref{lem:FaceIntegralAsymptotic} to bound
	\begin{equation*}
		\biggl(\int_{\cal F(\log B)} \prod_{i=1}^s (t_i + 1)^{\ell'_i} \,\mathrm d\mu\biggr)^{1 - \epsilon}
		\ll_{\epsilon} (\log B)^{(1-\epsilon)d + \sum_{i=1}^s (1-\epsilon)\ell'_i},
	\end{equation*}
	and $(1-\epsilon) \ell_i' = \ell_i$ by definition. Going back to~\eqref{eq:SmallVolume:1}, we obtain
	\begin{equation*}
		\int_{\Xi} \prod_{i=1}^s \mathrm e^{t_i} (t_i + 1)^{\ell_i} \,\mathrm d\vec t
		\ll_{A'} \biggl(\Bigl(\frac{\vol(\Xi)}{(\log B)^d}\Bigr)^\epsilon + (\log B)^{-A'}\biggr) B^a (\log B)^{d + \ell},
	\end{equation*}
	where $A'$ may also be chosen arbitrarily large.
\end{proof}

As an application of \cref{lem:Hoelder} we obtain the following estimate, which will deal with one of the error terms in \cref{lem:sumtointegral}.

\begin{lemma}\label{lem:boundary_theta_bound}
	Let $\vec \alpha \in \mathbb R_{> 0}^s$ and $\vec \varpi_i \in \mathbb R_{\geq  0}^s$ for $i = 1, \dots, k$. Let $\vec \ell$ be of convergent type for the face $\cal F$ and assume that $\cal F$ is not contained in a coordinate hyperplane. Moreover, let $\vec W = (\log (B/X_1), \ldots, \log (B/X_k))$ with sufficiently large $B$ and $X_i \ll (\log B)^N$ for some $N\in \mathbb{R}_{>1}$. Write
	\begin{equation*}
		\Xi = \cal P(\vec W + (\gamma \log \theta, \dots, \gamma \log \theta)) \setminus \cal P(\vec W - (\gamma \log \theta, \dots, \gamma \log \theta)).
	\end{equation*}
	Then for every $\delta'>0$ there exists $\epsilon >0$ such that for $1<\theta \leq 1 + (\log B)^{-\delta'}$ we have
\begin{equation*}
\int_{\Xi} \prod_{i=1}^s \mathrm e^{t_i} (1 + t_i)^{\ell_i} \,\mathrm d\vec t\ll_{\gamma} B^a (\log B)^{d+\ell-\epsilon}.
\end{equation*}
    
\end{lemma}

\begin{proof}
    Write $a_X\log B = \max_{\vec x \in \cal P(\vec W)} |\vec x|_1$ and $\Xi'=\Xi \setminus (a_X\log B - A \log \log B) \Delta^s$. 
    By definition of $\Xi$, any point in $\Xi$ has distance $\ll \log \theta$ to a point on the boundary $\partial \cal P(\vec W)$. Thus every point in $\Xi'$ has distance $\ll \log \theta$ to a point in $\partial \cal P(\vec W) \setminus (a_X\log B - 2 A \log \log B) \Delta^s$.

	The boundary $\partial \cal P(\vec W)$ is a union of simplices contained in a hyperplane. Notice that for sufficiently large $B$, any point in $\partial \cal P(\vec W) \setminus (a_X\log B - 2 A \log \log B) \Delta^s$ has distance $\ll A \log \log B$ to a point in $\cal F(\vec W)$. Put
	\begin{equation*}
		X_0 = \cal F(\vec W),\quad X_1 = \partial \cal P(\vec W) \setminus (a_X\log B - 2 A \log \log B) \Delta^s,\quad X_2 = \Xi'.
	\end{equation*}
	Then $\operatorname{diam}(X_0) \ll \log B$, $\operatorname{dist}(X_0, X_1) \ll \log \log B$ and $\operatorname{dist}(X_1, X_2) \ll \log \theta$. Splitting $X_1$ into parts contained in a hyperplane and $X_2$ accordingly, so that the distance bounds still hold, we may use \cref{lem:VolumeGrowthEstimate} to obtain
\begin{equation*}
		\operatorname{vol}(\Xi' ) \ll_{A} (\log B)^{d} (\log \log B)^{s-d-1} (\log \theta).
	\end{equation*}

We observe that $\Xi' \subset \cal P(\log B+\gamma \log \theta \mathbf{1})$. Hence, by \cref{lem:Hoelder} there exists a $\epsilon > 0$ such that
	\begin{equation*}
		\int_{\Xi'} \prod_{i=1}^s \mathrm e^{t_i} (t_i + 1)^{\ell_i} \,\mathrm d\vec t
		\ll_{A,\gamma} \biggl(\Bigl(\frac{\vol(\Xi')}{(\log B)^d}\Bigr)^\epsilon + (\log B)^{-A}\biggr) B^a (\log B)^{d + \ell},
	\end{equation*}
	where $A > 0$ may be chosen arbitrarily large. The lemma now follows by noting that
 \begin{equation*}
\int_{(a_X\log B - 2 A \log \log B) \Delta^s} \prod_{i=1}^s \mathrm e^{t_i} (1 + t_i)^{\ell_i} \,\mathrm d\vec t\ll e^{a\log B- 2 A \log \log B} (\log B)^{s+\sum_i \max(\ell_i,0)}
\end{equation*}
as $a_X\leq a$ and as we can take $A$ to be sufficiently large.
\end{proof}

\subsection{Towards an asymptotic formula}

The following technical lemma will provide asymptotic formulas for integrals over the standard simplex, which we can use to obtain asymptotic formulas for integrals over the polytopes $\cal P(\log B)$ after cutting them into simplices. Our goal is to evaluate integrals over the whole polytope $\cal P(\log B)$, for example for the main term of our counting function, by reducing them to integrals over the maximal face.

\begin{lemma}\label{lem:AsympSimplexD1}
	Let $f \colon \mathbb R^s \to \mathbb R$ be given by $f(\vec t) = \prod_{i=1}^s (t_i + 1)^{\ell_i}$ for any $\vec \ell \in \mathbb R^s$. Let $C \in \mathbb R_{\geq 0}$, $\vec c \in \mathbb R_{>0}^m$, $\vec x \in \mathbb R^s$ and $W \in \mathbb R^{s \times m}$ be a matrix such that $W \vec t + \vec x \in \mathbb R_{\geq 0}^s$ for all $\vec t \in C \Delta^m$. Then uniformly in all $C \geq 0$, we have
	\begin{multline*}
		\int_{C \Delta^m} \exp\Bigl(-\sum_{i=1}^m c_i t_i \Bigr) f(W \vec t + \vec x) \,\mathrm d\vec t = \frac{ f(\vec x)}{c_1 \cdots c_m}\\
		+ O_{W, \vec c}\Bigl( \mathrm e^{-\min_i c_i C} (C+1)^{m-1} \max_{\vec t \in C \Delta^m} |f(W \vec t + \vec x)| \Bigr)\\
		+ O_{W, \vec c}\biggl( \int_{C \Delta^m} \exp\Bigl(-\sum_{i=1}^m c_i t_i \Bigr) \tilde f(W \vec t + \vec x) \,\mathrm d\vec t \biggr),
	\end{multline*}
	where $\tilde f(\vec t) = f(\vec t) (\min_i t_i + 1)^{-1}$.
	
\end{lemma}

\begin{proof}
	We prove the lemma by induction on $m$. The base case $m = 1$ is obtained by an application of integration by parts, which yields
	\begin{equation*}
		\int_0^C \mathrm e^{-c t} f(\vec w t + \vec x) \,\mathrm dt
		= \Bigl[ -\frac{1}{c} \mathrm e^{-ct} f(\vec w t + \vec x) \Bigr]_0^C + \frac{1}{c} \int_0^C \mathrm e^{-ct} \mathrm D_{\vec w} f(\vec w t + \vec x) \,\mathrm dt.
	\end{equation*}
	The directional derivative $\mathrm D_{\vec w} f(\vec w t + \vec x)$ is bounded by $\tilde f(\vec w t + \vec x)$ up to a constant depending on $\vec \ell$ and $\vec w$. This immediately yields the desired result for $m = 1$.

	For the induction step, we assume that the lemma holds for $m - 1$ and prove it for $m$. Since the range of integration $C \Delta^m$ is symmetric in the $m$ variables, we may assume without loss of generality that $c_1 \leq c_2 \leq \cdots \leq c_m$. We integrate over $t_1$ last to see
	\begin{equation*}
		\int_0^C e^{-c_1t_1} \int_{(C - t_1) \Delta^{m-1}} \exp\Bigl(-\sum_{i=2}^m c_i t_i\Bigr) f(W \vec t + \vec x) \,\mathrm d(t_2,\ldots,t_m) \,\mathrm dt_1.
	\end{equation*}
	Using the induction hypothesis, we obtain
	\begin{multline}\label{eq:AsympSimplex:1}
		\int_{(C - t_1) \Delta^{m-1}} \exp\Bigl(-\sum_{i=2}^m c_i t_i\Bigr) f(W \vec t + \vec x) \,\mathrm d(t_2,\ldots,t_m)\\
		= \frac{f(\vec w_1 t_1 + \vec x)}{c_2 \cdots c_m} + O_{W, \vec c}( \mathrm e^{-c_2(C - t_1)} (C - t_1)^{m-2} \max_{\vec t \in (C - t_1) \Delta^{m-1}} |f(W \vec t + \vec x)| )\\
		+ O_{W, \vec c}( \int_{(C - t_1) \Delta^{m-1}} \exp\Bigl(-\sum_{i=2}^m c_i t_i\Bigr) \tilde f(W \vec t + \vec x) \,\mathrm d(t_2,\ldots,t_m) ),
	\end{multline}
	where $\vec w_1$ is the first column of $W$.

	After integrating over $t_1$, the first term may be resolved by an application of the base case $m = 1$. The only problematic term is the second error term
	\begin{equation*}
		\int_0^C e^{-c_1t_1} \tilde f(\vec w_1 t_1 + \vec x) \,\mathrm dt_1,
	\end{equation*}
	for which we need to artificially reintroduce the integration over $t_2, \ldots, t_m$.
	Since $\tilde f(\vec t) \ll f(\vec t) ((t_1 + 1)^{-1} + \cdots + (t_m + 1)^{-1})$, we can apply the induction hypothesis for each of the $m$ terms in the sum. Rearranging the resulting terms and estimating $\tilde{\tilde f}(\vec t) \leq \tilde f(\vec t)$ yields
	\begin{multline*}
		\tilde f(\vec w_1 t_1 + \vec x) \ll \int_{(C - t_1) \Delta^{m-1}} \exp\Bigl(-\sum_{i=2}^m c_i t_i\Bigr) \tilde f(W \vec t + \vec x) \,\mathrm d(t_2,\ldots,t_m)
		\\
		+ \mathrm e^{-c_2(C - t_1)} (C - t_1)^{m-2} \max_{\vec t \in (C - t_1) \Delta^{m-1}} |f(W \vec t + \vec x)|.
	\end{multline*}
	Hence, we find
	\begin{multline*}
		\int_0^C e^{-c_1t_1} \tilde f(\vec w_1 t_1 + \vec x) \,\mathrm dt_1
		\ll \int_{C \Delta^m} \exp\Bigl(-\sum_{i=1}^m c_i t_i\Bigr) \tilde f(W \vec t + \vec x) \,\mathrm d\vec t\\
		+ O_{W, \vec c}( \mathrm e^{-c_1 C} C^{m-1} \max_{\vec t \in C \Delta^m} |f(W \vec t + \vec x)| ),
	\end{multline*}
    which concludes the proof by returning to \cref{eq:AsympSimplex:1}.
\end{proof}

\begin{lemma}\label{lem:simplex_computation}
	Let $\vec \varpi, \vec \alpha$ and $\vec \ell$ be such that $\vec\ell$ is of convergent type for $\cal F$. Let $S \subset \cal P$ be a simplex, and write $d = \dim \cal F$.
	\begin{enumerate}[leftmargin=*, label=\normalfont(\roman*)]
	\item\label{enum:case-le-d} If at most $d$ vertices of $S$ are contained in $\cal F$, then
		\begin{equation*}
			\int_{(\log B) S} \prod_{i=1}^s \mathrm e^{t_i} (t_i + 1)^{\ell_i} \,\mathrm d\vec t \ll_{S} B^a (\log B)^{d + \ell - \delta},
		\end{equation*}
		where $\delta > 0$ is a constant that depends on $S$.
	\item\label{enum:d+1-case} If exactly $d+1$ vertices of $S$ are contained in $\cal F$, let $\cal F_S = \cal F \cap S$ be the face of $S$ contained in $\cal F$. Then there exists a constant $C_{S}$ not depending on $\vec \ell$ such that
		\begin{align*}
			\int_{(\log B) S} \prod_{i=1}^s \mathrm e^{t_i} (t_i + 1)^{\ell_i} \,\mathrm d\vec t &= C_{S} B^a \int_{(\log B)\cal F_S} \prod_{i=1}^s (t_i + 1)^{\ell_i} \,\mathrm d\mu\\
			& \qquad \qquad\qquad+ O_{S}(B^a (\log B)^{d + \ell - \delta}),
		\end{align*}
		where $\delta > 0$ is a constant that depends on $S$.
	\item\label{enum:C_S} Let $\vec v_0, \ldots, \vec v_s$ be the vertices of $S$ and $\vec v_0, \ldots, \vec v_d$ those contained in $\cal F$. Put $\vec w_i = \vec v_i - \vec v_0$, and let $\operatorname{vol}_d(\cal F_S)$ be the $d$-dimensional volume of $\cal F_S$. Then the constant $C_S$ in \ref{enum:d+1-case} is 
		\begin{align*}
			C_S &= \frac{\det(\vec w_1, \ldots, \vec w_s)}{\sqrt{\det((\vec w_1, \ldots, \vec w_d)^t (\vec w_1, \ldots, \vec w_d))}} \prod_{i=d+1}^s \Bigl(-\sum_{j=1}^s w_{ij}\Bigr)^{-1}\\
			&= \frac{s! \operatorname{vol}(S)}{d! \operatorname{vol}_d(\cal F_S)} \prod_{i=d+1}^s \bigl(a - |\vec v_i|_1\bigr)^{-1}.
		\end{align*}

	\end{enumerate}
\end{lemma}

\begin{proof}
	We begin with some set-up steps that are useful for both parts of the lemma. Write $\vec v_0, \ldots, \vec v_s$ for the vertices of $S$ and assume without loss of generality that the first $m+1$ vertices $\vec v_0, \ldots, \vec v_m$ are those that are contained in $\cal F$. Write
	\begin{equation*}
		\vec w_i = \vec v_i - \vec v_0 \quad (1 \leq i \leq s) \quad\text{and}\quad c_i = -\sum_{j=1}^s w_{ij} \quad (m < i \leq s).
	\end{equation*}
	When at least one vertex $\vec v_0$ of $S$ is contained in $\cal F$, we have $|\vec v_0|_1 = a$ and $|\vec v_i|_1 < a$, so $c_i = |\vec v_0|_1 - |\vec v_i|_1 > 0$ for all $i > m$.
	The affine map 
	\begin{equation*}
		F\colon\mathbb R^s \to \mathbb R^s, \quad \vec u \mapsto (\log B) \vec v_0 + \sum_{i=1}^s u_i \vec w_i
	\end{equation*}
	has Jacobian determinant $J_S = |\det(\vec w_1, \ldots, \vec w_s)|$ and sends $(\log B) \Delta^s$ bijectively onto $(\log B) S$. A change of variables $\vec t = F(\vec u)$ yields
	\begin{equation}\label{eq:SetupForm}
		I \coloneq \int_{(\log B) S} \prod_{i=1}^s \mathrm e^{t_i} (t_i + 1)^{\ell_i} \,\mathrm d\vec t
		= J_S B^a \int_{(\log B) \Delta^s} \prod_{i=m+1}^s\mathrm e^{-c_i u_i} \bigl( F_i(\vec u) + 1 \bigr)^{\ell_i} \,\mathrm d\vec u.
	\end{equation}

	\subsubsection*{Proof of \ref{enum:case-le-d}} By assumption, we have $m \leq d - 1$. Let $A > 0$ be a constant to be chosen later and split $(\log B) \Delta^s = \cal U^{-} \cup \cal U^{+}$ with
	\begin{align*}
		\cal U^{-} &= \Bigl\{ \vec u \in (\log B) \Delta^s\ :\ c_{m+1} u_{m+1} + \cdots + c_s u_s \geq A \log \log B \Bigr\}\ \text{and}\\
		\cal U^{+} &= \Bigl\{ \vec u \in (\log B) \Delta^s\ :\ c_{m+1} u_{m+1} + \cdots + c_s u_s < A \log \log B \Bigr\}.
	\end{align*}
	On $\cal U^{-}$ we have $|\vec t|_1 \leq a(\log B) - A \log \log B$, so $\prod_{i=1}^s \mathrm e^{t_i} \leq B^a (\log B)^{-A}$ and $(F(\vec u) + \vec 1)^{\vec \ell} \ll (\log B)^{|\vec \ell|_1}$, while $\mathrm{vol}(\cal U^{-}) \leq \mathrm{vol}((\log B) \Delta^s) \ll (\log B)^s$. Hence, the contribution of $\cal U^{-}$ to \eqref{eq:SetupForm} is
	\begin{multline*}
		J_S B^a \int_{\cal U^{-}} \prod_{i=m+1}^s\mathrm e^{-c_i u_i} \bigl( F_i(\vec u) + 1 \bigr)^{\ell_i} \,\mathrm d\vec u
		\ll B^a (\log B)^{-A} (\log B)^{|\vec \ell|_1} (\log B)^s\\
		= B^a (\log B)^{s + |\vec \ell|_1 - A},
	\end{multline*}
	which is $\ll B^a (\log B)^{d + \ell - \delta}$ provided that $A > s - d + \delta + |\vec \ell|_1 - \ell$.

	The contribution of $\cal U^{+}$ can be handled using only its volume. By construction, we have
	\begin{equation*}
		\cal U^{+} \subset (\log B) \Delta^m \times \prod_{i=m+1}^s [0, c_i^{-1} A \log \log B],
	\end{equation*}
	which yields $\operatorname{vol}(\cal U^{+}) \ll_{\vec c,A} (\log B)^m (\log \log B)^{s-m}$. Hence, the volume of $F(\cal U^{+})$ is bounded by the same function, with the implied constant multiplied by $J_S^{-1}$. The integration corresponding to $\cal U^{+}$ is now easily bounded using \cref{lem:Hoelder}, which gives
	\begin{equation}\label{eq:SimplexComputation:Uplus}
		\int_{F(\cal U^{+})} \prod_{i=1}^s \mathrm e^{t_i} (t_i + 1)^{\ell_i} \,\mathrm d\vec t \ll \Bigl((\log B)^{\delta(m-d)} + (\log B)^{-A} \Bigr) B^a (\log B)^{d + \ell}
	\end{equation}
	for some $\delta > 0$. Together with the bound on $\cal U^{-}$, this completes the proof of \ref{enum:case-le-d}.

	\subsubsection*{Proof of \ref{enum:d+1-case}} Now $m = d$ and $S$ has exactly $d+1$ vertices $\vec v_0, \ldots, \vec v_d$ that are contained in $\cal F$. Similarly to the previous part, we split up the domain $(\log B) \Delta^s$. This time, we introduce three different regions. The `new' region is used to make sure that the errors obtained by an application of \cref{lem:AsympSimplexD1} are sufficiently small. Write $\vec u' = (u_1, \ldots, u_d)$ for the first $d$ coordinates of $\vec u$. Define
	\begin{align*}
		\cal U^{+}&= \Bigl\{ \vec u \in (\log B) \Delta^s : |\vec u'|_1 \leq \log B - A \log \log B \Bigr\},\\
		\cal U^{*}&= \Bigl\{ \vec u \in (\log B) \Delta^s : |\vec u'|_1 > \log B - A \log \log B,\ u_i \leq c_i^{-1} A \log \log B\ \forall i > d \Bigr\},\\
		\cal U^{-}&= \Bigl\{ \vec u \in (\log B) \Delta^s : u_i > c_i^{-1} A \log \log B \text{ for some } i > d \Bigr\}.
	\end{align*}
	Then $(\log B) \Delta^s = \cal U^{+} \cup\, \cal U^{*} \cup\, \cal U^{-}$.

	Only $\cal U^{+}$ will be evaluated asymptotically, while $\cal U^{*}$ and $\cal U^{-}$ will simply be bounded.
	For the last of these, the argument used in part~\ref{enum:case-le-d} to bound the contribution of $\cal U^{-}$ works again, provided that $A$ is chosen sufficiently large.
	The estimate for $\cal U^{*}$ is slightly more delicate. Notice that all points in $\cal U^{*}$ are at distance $\ll_{\vec c, A} \log \log B$ from a point satisfying $u_1 + \cdots + u_d = \log B$ and $u_i = 0$ for $i > d$. Putting
	\begin{equation*}
		X_0 = \{ \vec u \in (\log B) \Delta^s : u_1 + \cdots + u_d = \log B,\ u_i = 0\ \forall i > d \},\
	\end{equation*}
		$X_1 = X_0$, and $X_2 = \cal U^{*}$, we have $\operatorname{dist}(X_0, X_1) = 0$ and $\operatorname{dist}(X_1, X_2) \ll_{\vec c, A} \log \log B$. By definition $X_0$ is contained in a $(d-1)$-dimensional affine subspace. Hence, we can apply \cref{lem:VolumeGrowthEstimate} with $\delta_0 = \log B$, $\delta_1 = 0$ and $\delta_2 = \log \log B$ to obtain
	\begin{equation*}
		\operatorname{vol}(\cal U^{*}) \ll_{\vec c, A} (\log B)^{d-1} (\log \log B)^{s-d+1}.
	\end{equation*}
	A sufficient upper bound for the integral over $\cal U^*$ is now obtained via \cref{lem:Hoelder} exactly like the estimate in~\eqref{eq:SimplexComputation:Uplus}.

	Finally, we turn to the region $\cal U^{+}$, which will contribute to the main term. Start by decomposing
	 \begin{multline}\label{eq:SimplexComputation:main}
		 J_S B^a \int_{\cal U^{+}} \prod_{i=d+1}^s\mathrm e^{-c_i u_i} \bigl( F_i(\vec u) + 1 \bigr)^{\ell_i} \,\mathrm d\vec u\\
		 = J_S B^a \int_{(\log B - A \log \log B) \Delta^d} \int_{(\log B - |\vec u'|_1) \Delta^{s-d}} \prod_{i=d+1}^s\mathrm e^{-c_i u_i} \bigl( F_i(\vec u) + 1 \bigr)^{\ell_i} \,\mathrm d\vec u'' \,\mathrm d\vec u'.
	 \end{multline}
	Fixing $\vec u' = (u_1, \ldots, u_d) \in (\log B - A \log \log B) \Delta^d$ to estimate the inner integral, we apply \cref{lem:AsympSimplexD1}, with the matrix $W$ whose columns are $\vec w_{d+1}, \ldots, \vec w_s$, $f(\vec x) = (\vec x + \vec 1)^{\vec \ell}$, and $\vec x = (\log B)\vec v_0 + \sum_{i=1}^d u_i \vec w_i$. Thus
	\begin{multline}\label{eq:InnerAsymp}
		\int_{(\log B - |\vec u'|_1) \Delta^{s-d}} \exp\Bigl(-\sum_{i=d+1}^s c_i t_i\Bigr) f(W\vec t + \vec x) \,\mathrm d \vec t
		= \frac{f(\vec x)}{c_{d+1} \cdots c_s}\\
		+ O_{W, \vec c}\Bigl( \mathrm e^{-\min_{d<i\leq s} c_i (\log B - |\vec u'|_1)} (\log B)^{s-d-1} \max_{\vec t \in (\log B) \Delta^{s-d}} |f(\vec t)| \Bigr)\\
		+ O_{W, \vec c}\Bigl( \int_{(\log B - |\vec u'|_1) \Delta^{s-d}} \mathrm e^{-\sum_{i=d+1}^s c_i u_i} \tilde f(W\vec t + \vec x) \,\mathrm d\vec t \Bigr),
	\end{multline}
	where $\tilde f(\vec x) = f(\vec x)(\min_{d<i\leq s} x_i + 1)^{-1}$. We denote the two error terms in \eqref{eq:InnerAsymp} by $E_1(\vec u')$ and $E_2(\vec u')$, respectively.

	We begin with a look at $E_1$. Without loss of generality, we may assume that $c_{d+1} = \min_{d<i\leq s} c_i > 0$.
	On the main region, $\log B - |\vec u'|_1 \geq A \log \log B$, so $\mathrm e^{-c_{d+1}(\log B - |\vec u'|_1)}$ is at most $(\log B)^{-c_{d+1} A}$, and $\max_{\vec u''} |F(\vec u', \vec u'')| \ll (\log B)^{|\vec \ell|_1}$.
	Integrating $E_1(\vec u')$ over $\vec u' \in (\log B - A \log \log B) \Delta^d$, we find that the contribution of $E_1$ to \cref{eq:SimplexComputation:main} is
	\begin{equation*}
		\ll B^a (\log B)^{d} \cdot (\log B)^{-c_{d+1} A} (\log B)^{s-d-1} (\log B)^{|\vec \ell|_1} \ll B^a (\log B)^{d + \ell - \delta},
	\end{equation*}
	provided $A$ is chosen sufficiently large. 

	To estimate the contribution of $E_2$, we shall use \cref{lem:Hoelder}. Overestimating the integration over $\vec u'$ by the full simplex, the contribution of $E_2$ to the integral in \eqref{eq:SimplexComputation:main} is
	\begin{equation*}
		\ll \int_{(\log B) \Delta^s} \mathrm e^{-\sum_{i=d+1}^s c_i u_i} \tilde f(F(\vec u)) \,\mathrm d\vec u.
	\end{equation*}
	Reversing the change of variables made in \cref{eq:SetupForm}, this contribution is bounded by
	\begin{equation*}
		\int_{(\log B) S} \prod_{i=1}^s \mathrm e^{t_i} (t_i + 1)^{\ell_i} (\min\nolimits_i t_i + 1)^{-\delta} \,\mathrm d\vec t,
	\end{equation*}
	for a small $\delta > 0$ (since the weight $(\min_{d<i\leq s} u_i + 1)^{-1}$ is $\ll (\min_{i} t_i + 1)^{-\delta}$ on the polytope for $\delta$ small). Because being of convergent type is an open condition, $\vec \ell - \delta \vec e_i$ is still of convergent type for all standard basis vectors $\vec e_i$ and $\delta > 0$ sufficiently small. Hence, a combination of \cref{lem:UpperBoundNoX} and \cref{lem:FaceIntegralAsymptotic} bounds the contribution of $E_2$ to \cref{eq:SimplexComputation:main} by
	\begin{equation*}
		\ll_{S} B^a (\log B)^{d + \ell - \delta}.
	\end{equation*}

	Lastly, we have to estimate the main term. Combining the main term and the estimates above, we obtain
	\begin{multline}\label{eq:MainTerm}
		I = \frac{B^a J_S}{c_{d+1} \cdots c_s}  \int_{(\log B - A \log \log B) \Delta^d} \bigl( F(\vec u', \vec 0) + \vec 1 \bigr)^{\vec \ell} \,\mathrm d\vec u'\\
		+ O\bigl( B^a (\log B)^{d + \ell - \delta} \bigr).
	\end{multline}
	We now convert the remaining integral into an integral over the face. Since being of convergent type is an open condition, we may find a small $\delta > 0$ such that $\vec \ell' = \frac{1}{1 - \delta} \vec \ell$ is still of convergent type. Split
	\begin{multline*}
		\int_{(\log B - A \log \log B) \Delta^d} \bigl( F(\vec u', \vec 0) + \vec 1 \bigr)^{\vec \ell} \,\mathrm d\vec u'
		= \int_{(\log B) \Delta^d} \bigl( F(\vec u', \vec 0) + \vec 1 \bigr)^{\vec \ell} \,\mathrm d\vec u'\\
		- \int_{(\log B) \Delta^d \setminus (\log B - A \log \log B) \Delta^d} \bigl( F(\vec u', \vec 0) + \vec 1 \bigr)^{\vec \ell} \,\mathrm d\vec u'.
	\end{multline*}
	Hölder's inequality with exponents $p = (1-\delta)^{-1}$ and $q = \delta^{-1}$ bounds the second term by
	\begin{equation*}
		\Bigl( \int_{(\log B) \Delta^d} \bigl( F(\vec u', \vec 0) + \vec 1 \bigr)^{\frac{1}{1-\delta}\vec \ell} \,\mathrm d\vec u' \Bigr)^{1-\delta}
		\Bigl( \mathrm{vol}\bigl( (\log B) \Delta^d \setminus (\log B - A \log \log B) \Delta^d \bigr) \Bigr)^{\delta}.
	\end{equation*}
	Applying \cref{lem:FaceIntegralAsymptotic} to the first factor (after mapping $(\log B) \Delta^d$ onto the face $\cal F_S$; note that this lemma is stated for the full face $\cal F$, but the same argument applies to the subface $\cal F_S$, whose vertices are $d$-dimensional subsets of those of the maximal face) and estimating the volume of the boundary layer by $\ll (\log B)^{d-1} \log \log B$, we find
	\begin{multline*}
		\int_{(\log B) \Delta^d \setminus (\log B - A \log \log B) \Delta^d} \bigl( F(\vec u', \vec 0) + \vec 1 \bigr)^{\vec \ell} \,\mathrm d\vec u'\\
		\ll (\log B)^{(1-\delta)(d + \sum_i \ell'_i)} \cdot \bigl( (\log B)^{d-1} \log \log B \bigr)^{\delta}
		\ll (\log B)^{d + \ell - \delta + \varepsilon}.
	\end{multline*}
	This term is therefore absorbed into the error in \eqref{eq:MainTerm} after shrinking $\delta$ to absorb the $\varepsilon$, and we are left with
	\begin{equation}\label{eq:MainTerm2}
		I = \frac{J_S}{c_{d+1} \cdots c_s} B^a \int_{(\log B) \Delta^d} \bigl( F(\vec u', \vec 0) + \vec 1 \bigr)^{\vec \ell} \,\mathrm d\vec u'
		+ O\bigl( B^a (\log B)^{d + \ell - \delta} \bigr).
	\end{equation}

	The affine map $\vec u' \mapsto \vec t = (\log B)\vec v_0 + \sum_{i=1}^d u_i \vec w_i$ sends $(\log B) \Delta^d$ bijectively onto $(\log B)\cal F_S$ with Jacobian $J_S'$. Here $J_S'$ is a positive constant depending only on $S$, not on $\vec \ell$ or $B$; explicitly, writing $W_d$ for the $s \times d$ matrix with columns $\vec w_1, \ldots, \vec w_d$, the $d$-dimensional Euclidean measure on the face satisfies $J_{S}' \mathrm d\mu = \,\mathrm d\vec u'$ with $J_S' = \bigl(\sqrt{\det(W_d^T W_d)}\bigr)^{-1}$. Under this map the integrand becomes $\prod_{i=1}^s (t_i + 1)^{\ell_i}$, so that
	\begin{equation*}
		\int_{(\log B) \Delta^d} \bigl( F(\vec u', \vec 0) + \vec 1 \bigr)^{\vec \ell} \,\mathrm d\vec u'
		= J_S' \int_{(\log B) \cal F_S} \prod_{i=1}^s (t_i + 1)^{\ell_i} \,\mathrm d\mu.
	\end{equation*}
	Inserting this into \eqref{eq:MainTerm2}, we obtain
	\begin{equation*}
		I = \frac{J_S J_S'}{c_{d+1} \cdots c_s} B^a \int_{(\log B) \cal F_S} \prod_{i=1}^s (t_i + 1)^{\ell_i} \,\mathrm d\mu
		+ O\bigl( B^a (\log B)^{d + \ell - \delta} \bigr).
	\end{equation*}
	The claim follows with
	\begin{equation*}
		C_S = \frac{J_S J_S'}{c_{d+1} \cdots c_s},
	\end{equation*}
	which depends only on the simplex $S$ and not on $\vec \ell$.
	
	\subsubsection*{Proof of~\ref{enum:C_S}} Simply use that the volume of the $d$-dimensional standard simplex is $1/d!$ and that $\vec w_i = \vec v_i - \vec v_0$.
\end{proof}

\begin{theorem}\label{maintheorem_with_simplices}
    Let $f\colon \mathbb N^s \to \mathbb C$ satisfy \cref{mainassumption} and be real and non-negative or satisfy \cref{upperboundassp} as well. Assume that the polytope $\cal P$ is bounded, that $\cal F$ is not contained in any coordinate hyperplane, and that $\vec \ell$ is of convergent type for $\cal F$. 
    Dissect the polytope $\cal P$ into simplices and let $S_1,\ldots, S_h$ be those simplices that have exactly $d+1$ vertices in the maximal face $\cal F$, where $d = \dim \cal F$. Let the constants $C_{S_j}$ be as in \cref{lem:simplex_computation} and $\cal F_{S_j} = \cal F \cap S_j$ for $1\leq j\leq h$.
	
	Then there exists $\varepsilon > 0$ such that 
    \begin{align*}
\Upsilon(B; \vec 1)
		= c_{\cal P} c_f \Bigl( \prod_{i=1}^s \alpha_i^{-\ell_i} \Bigr) B^a (\log B)^{d+\ell}  +  O_{\vec \varpi, \vec \alpha, \vec \ell}(C_f B^a (\log B)^{d + \ell - \varepsilon}),
\end{align*}
where
\begin{align}\label{def:cP}
c_{\cal P} = \sum_{j=1}^h C_{S_j} \int_{\cal F_{S_j}} \prod_{i=1}^s t_i^{\ell_i} \,\mathrm d\mu.
\end{align}
\end{theorem}

\begin{proof}
Let $\gamma$ be as in \cref{lem:sumtointegral}. Then by \cref{lem:sumtointegral} we have
\begin{multline*}
		\Upsilon(B; \vec 1)
		= c_f \Bigl( \prod_{i=1}^s \alpha_i^{-\ell_i} \Bigr) \int_{\cal P(\vec W)}  \prod_{i=1}^s \mathrm e^{t_i} (t_i + 1)^{\ell_i} \,\mathrm d\vec t\\
		+ O\biggl(
		C_f (\log \theta)^{-A - s} \int_{\cal P(\vec W + (\gamma \log \theta,\ldots,\gamma\log\theta))} (1 + \min\nolimits_i t_i)^{-\delta} \prod_{i=1}^s \mathrm e^{t_i} (1 + t_i)^{\ell_i} \,\mathrm d\vec t
		\biggr)\\
		+ O\biggl(
		C_f \int_{\Xi} \prod_{i=1}^s \mathrm e^{t_i} (1 + t_i)^{\ell_i} \,\mathrm d\vec t
		\biggr),
	\end{multline*}
	where
	\begin{equation*}
		\Xi = \cal P(\vec W + (\gamma \log \theta, \dots, \gamma \log \theta)) \setminus \cal P(\vec W - (\gamma \log \theta, \dots, \gamma \log \theta)).
	\end{equation*}
Assume that $\theta= 1+ (\log B)^{-\delta'}$ for some $\delta'>0$. Then by \cref{lem:boundary_theta_bound}, there exists $\delta>0$ such that
\begin{equation*}
\int_{\Xi} \prod_{i=1}^s \mathrm e^{t_i} (1 + t_i)^{\ell_i} \,\mathrm d\vec t\ll_{\gamma} B^a (\log B)^{d+\ell-\delta}. 
\end{equation*}

Next we apply \cref{lem:UpperBoundNoX} and find the upper bound
\begin{align*}
\int_{\cal P(\vec W + (\gamma \log \theta,\ldots,\gamma\log\theta))} (1 + \min\nolimits_i t_i)^{-\delta} \prod_{i=1}^s \mathrm e^{t_i} (1 + t_i)^{\ell_i} \,\mathrm d\vec t \\ \ll_{\gamma} e^{a \left(\log B+ \gamma \log \theta\right)} (\log B)^{d+\ell - \delta}\ll_{\gamma} B^a (\log B)^{d+\ell-\delta}.
\end{align*}
By taking $\delta'>0$ sufficiently small, we can ensure that there exists $\delta''>0$ such that
\begin{align*}
C_f (\log \theta)^{-A - s} \int_{\cal P(\vec W + (\gamma \log \theta,\ldots,\gamma\log\theta))} (1 + \min\nolimits_i t_i)^{-\delta} \prod_{i=1}^s \mathrm e^{t_i} (1 + t_i)^{\ell_i} \,\mathrm d\vec t \\ \ll C_f B^a (\log B)^{d+\ell-\delta''}.
\end{align*}
Applying \cref{lem:simplex_computation} leads to

		\begin{multline*}
			\int_{\cal P(\vec W)}  \prod_{i=1}^s \mathrm e^{t_i} (t_i + 1)^{\ell_i} \,\mathrm d\vec t = \sum_{j=1}^h C_{S_j} B^a \int_{(\log B)\cal F_{S_j}} \prod_{i=1}^s (t_i + 1)^{\ell_i} \,\mathrm d\mu\\
			+ O(B^a (\log B)^{d + \ell - \delta}),
		\end{multline*}
for some $\delta >0$. Together with \cref{lem:FaceIntegralAsymptotic} applied to the subfaces instead of the whole maximal face, we obtain 
\begin{align*}
\Upsilon(B; \vec 1)
		= c_f \Bigl( \prod_{i=1}^s \alpha_i^{-\ell_i} \Bigr) B^a (\log B)^{d+\ell}\sum_{j=1}^h C_{S_j} \int_{\cal F_{S_j}} \prod_{i=1}^s t_i^{\ell_i} \,\mathrm d\mu\\
			+ O(C_f B^a (\log B)^{d + \ell - \delta})
\end{align*}
for some $\delta >0$.
\end{proof}

\section{Families over toric bases}\label{sec:proofMT_examples}
We now turn to proving \cref{main_theorem_toric} and providing two classes of examples of specific toric varieties together with inputs that are `admissible' for each.
Keep all notation as in the introduction and recall that 
\begin{equation*}
	H(\vec x) = \max_{\sigma\in \Sigma^{\max}} |\vec x^{-K(\sigma)}|,
\end{equation*}
where $-K = \sum_{d\in \Dc, 1\le j\le a_d} D_{d,j}$ is anticanonical, is an anticanonical height function.
The set $\{-K(\sigma) : \sigma\in \Sigma^{\max}\}$ of divisors consists precisely of the extremal ways to write the class $[-K]$ as a non-negative linear combination of the boundary divisors, meaning that each of the entries in the maximum corresponds to a vertex of the polytope $\RR_{\ge 0}^{\Sigma(1)} \cap \deg^{-1}(-K)$, where $\deg\colon \ZZ^{\Sigma(1)}\to \Pic X$ maps each ray to the class of its corresponding divisor.

\begin{lemma}\label{lem:factorization_H_theta}
	The height $H\colon \ZZ^{n+\rho}\to\RR_{>0}$ factors through
	\begin{equation*}
	\norm{\cdot}\colon \ZZ^{n+\rho} \to \ZZ^{\Dc}_{\ge 0},\ \vec x \mapsto (\norm{x_d}_\infty)_{d\in \Dc}
	\end{equation*}
	and the coprimality condition $\theta$, detecting \cref{eq:coprimality-toric}, factors through
	\begin{equation*}
		g\colon \ZZ^{n+\rho} \to \ZZ^{\Dc}_{>0},\ \vec x \mapsto (\gcd(\vec x_d))_{d\in \Dc}.
	\end{equation*}
\end{lemma}
\begin{proof}
This simply means that both the coprimality condition and the height are symmetric within the $x_{d,j}$ for fixed $d\in \Dc$. To verify this, note that each linear equivalence $D_{d,j}\sim D_{d,k}$ between divisors corresponding to variables $x_{d,j}$ and $x_{d,k}$ is induced by a character $\chi_m$, meaning that the generators of the two corresponding rays satisfy $\langle m, u_{d,k}\rangle = - \langle m, u_{d,j}\rangle$, while all other ray generators are contained in the hyperplane $H_m$ orthogonal to $m$. Thus for each maximal cone $\sigma$ containing $\rho_{d,j}$, its facet $\sigma_0 = \Cone(\sigma(1)\setminus \rho_{d,j})$ is contained in $H_m$. Consequently, $\sigma' = \Cone(\sigma_0 \cup \rho_{d,k})$ is another maximal cone, inducing a coprimality condition that precisely swaps the roles of $x_{d,j}$ and $x_{d,k}$. Moreover, $-K(\sigma')$ replaces the summand $k_{d,j} D_{d,j}$ of $-K(\sigma)$ by $k_{d,j} D_{d,k}$.
\end{proof}
From this, we obtain a function $\theta'\colon \ZZ_{\ge 0}^\Dc \to \{0,1\}$ tracking the coprimality condition and a height $H'\colon \ZZ^{\Dc} \to \RR_{>0}$ given by $k$ monomials, that is, there are vectors $\vec\varpi_1,\dots,\vec\varpi_k\in \ZZ_{\ge 0}^\Dc$ such that~\eqref{eq:package-height} holds:
\begin{equation*}
	H'(\vec b) = \max_{1\le j\le k}\vec b^{\vec\varpi_j} = \max_{1\le j\le k} \prod_{d\in \Dc} b_d^{\varpi_{j,d}}.
\end{equation*}
We may rewrite the number $N(B)$ of everywhere locally soluble fibres as
\begin{equation*}
	N(B) = 2^{-\rho} \sum_{\substack{\vec x\in \ZZ^{n+\rho}\\ H(\vec x)\le B}} \theta(\vec x) \prod_{i\in \Ic} \sigma_i(\vec x_i),
\end{equation*}
where
\begin{equation*}
	\theta(\vec x) = \begin{cases}
		1, & \gcd(x_{d,j} : \rho_{d,j}\not\subset\sigma) =1 \text{ for all } \sigma\in \Sigma^{\max}, \\
		0, & \text{else}
	\end{cases}
\end{equation*}
and
\begin{equation*}
	\sigma_i(\vec x_i) = \begin{cases}
		1, & \psi_i^{-1}(\vec x_i)(\mathbb{A}_\QQ)\ne \emptyset \\
		0, & \text{else}.
	\end{cases}
\end{equation*}

\subsection{Möbius inversion} As a first simplification, let $\mu\colon \ZZ_{>0}^{n+\rho}\to \ZZ$ be the Möbius function introduced by Salberger~\cite[11.9]{MR1679841}.
\begin{lemma}\label{lem:mobius-factorization}
	This Möbius function factors through the diagonal embedding $\ZZ^\Dc\to \ZZ^{n+\rho}$, in the sense that $\mu(\vec \eta)=0$ whenever $\vec \eta$ does not lie in its image.
\end{lemma}
\begin{proof}
	Using \cref{lem:factorization_H_theta}, we factor $\theta = \theta'\circ g$. The resulting function $\theta'\colon \ZZ_{>0}^{\Dc}\to \{0,1\}$ induces a Möbius function $\mu'\colon \ZZ_{>0}^{\Dc}\to \ZZ$ by Salberger's construction, satisfying
	\begin{equation*}
		\theta(\vec x) = \theta'(g(\vec x)) = \sum_{\vec \eta \mid g(\vec x)}\mu'(\vec \eta) = \sum_{\substack{\vec \eta \in \ZZ^{\Dc} \\ \eta_d \mid \vec x_d}} \mu'(\vec \eta).
	\end{equation*}
	Extending this new Möbius function by $0$, it coincides with $\mu$ by uniqueness of the latter.
\end{proof}
We arrive at
\begin{equation*}
	N(B) = \sum_{\vec \eta \in \ZZ_{>0}^{\Dc} } \mu'(\vec\eta)N(\vec \eta, B),
\end{equation*}
where
\begin{equation*}
	N(\vec\eta, B) = \sharp\{\vec x \in \ZZ^{n+\rho} : H(\vec x)\le B,\  \eta_d \mid \vec x_d, \text{ and } \vec x_d\neq 0 \text{ for all } d\in \Dc \}.
\end{equation*}
For $\vec b \in \RR_{>0}^{\Dc}$ and $\vec\eta\in \ZZ_{>0}^{\Dc}$, define
\begin{equation*}
	n_{\vec\eta}(\vec b) = 2^{-\rho} \sum_{\substack{\vec x\in \ZZ^{n+\rho}\\ \norm{\vec x_d} = b_d\\ \eta_d\mid\vec x_d}} \prod_{d\in \Dc} \sigma_d(\vec x_d),
\end{equation*}
where $\norm{\cdot}$ denotes the supremum norm and we set $\sigma_d(\vec x_d)=1$ for $d\in \Dc\setminus\Ic$. Appealing to \cref{lem:factorization_H_theta}, we factor the height function $H = H'\circ \norm{\cdot}$ through $\vec x \mapsto (\norm{x_d})_d$ in order to get
\begin{equation*}
	N(B) = \sum_{\vec \eta\in \ZZ_{>0}^{\Dc}} \mu'(\vec \eta)N(\vec\eta, B) 
\end{equation*}
with
\begin{equation*}
	N(\vec\eta, B) = \sum_{\substack{\vec b \in \ZZ_{>0}^{\Dc} \\ H'(\vec b)\le B}} n_{\vec \eta}(\vec b)
\end{equation*}
and $H'$ as in~\eqref{eq:package-height}.
In order to apply \cref{maintheorem}, for $1\leq \eta_d\leq B_d$, $d\in \Dc$, we need to compute sums of the form
\begin{align*}
	\sum_{\vec b\leq \vec B} n_{\vec \eta}(\vec b) = 2^{-\rho} \sum_{\substack{\vec x\in \ZZ^{n+\rho} \\ \norm{\vec x_d} \leq B_d/\eta_d}} \prod_{d\in \Dc} \sigma_d(\vec x_d) = 2^{-\rho} \prod_{d\in \Dc} N_d\!\left(\frac{B_d}{\eta_d}\right).
\end{align*}
\begin{lemma}
	If $B_d>1$ and $\eta_d\ge 1$, then
 \begin{equation}\label{eq:N(B/eta)}
	N_d\!\left(\frac{B_d}{\eta_d}\right) = 
	c_d\left(\frac{B_d}{\eta_d}\right)^{a_d} (1+\log B_d)^{-\Delta_d} (1 + O((1+ \log \eta_d)(1+\log B_d)^{-\delta})).
 \end{equation}
\end{lemma}
\begin{proof}
If $\log \eta_d\le (1+ \log B_d)/2$, we may use~\eqref{eq:affine-asymp} together with
\begin{equation*}
	(1+\log B_d - \log \eta_d)^{-\Delta_d} = (1+\log B_d)^{-\Delta_d} +O ((\log \eta_d )(1+\log B_d)^{-\Delta_d-1}).
\end{equation*}
If $(1+ \log B_d)/2 < \log \eta_d$, note that the error term on the right-hand side of~\eqref{eq:N(B/eta)} dominates. If $\eta_d \le B_d$, then appealing to~\eqref{eq:affine-asymp} yields
\begin{equation*}
	N_d\!\left(\frac{B_d}{\eta_d}\right) \ll  \left(\frac{B_d}{\eta_d}\right)^{a_d} (1+\log B_d)^{-\Delta_d},
\end{equation*}
while $N_d(B_d/\eta_d) = 0$ as soon as $\eta_d>B_d$,
and in each case $N_d$ is thus sufficiently small.
\end{proof}
We arrive at
\begin{align*}
	\sum_{\vec b\leq \vec B} n_{\vec \eta}(\vec b) = 2^{-\rho} \prod_{d\in \Dc} c_d \left(\frac{B_d}{\eta_d}\right)^{a_d} (1+\log B_d)^{-\Delta_d}\left(1+O\!\left(\max_{d\in \Dc} \frac{1+\log \eta_d}{(1+\log B_d)^{\delta}}\right)\right)
\end{align*}
and have thus verified assumption~\eqref{mainassumption} with the parameters $s=|\Dc|$, $\alpha_d=a_d$, $\ell_d=-\Delta_d$,
\begin{equation*}
c_f= 2^{-\rho} \prod_{d\in \Dc} \frac{c_d}{\eta_d^{a_d}},
\end{equation*}
and $C_f = c_f\max_{d\in \Dc} (1+ \log \eta_d )$.
We deduce that
\begin{equation}\label{eq:applied-main-thm}
	N(\vec\eta, B)= c_{\cal P} 2^{-\rho} \prod_{d\in \Dc} \frac{c_d a_d^{\Delta_d}}{\eta_d^{a_d}} B (\log B)^{\rho-1-\Delta} \left(1+O\!\left( \max_{d\in \Dc}\frac{1+\log \eta_d}{(1+\log B)^{\vartheta}}\right)\right),
\end{equation}
where
\begin{equation*}
	\Delta = \sum_{d\in \Dc} \Delta_d
\end{equation*}
and $c_{\cal P}$ is as in \cref{def:cP} and $\vartheta >0$.
Indeed, that the maximal face is of dimension $\rho-1$ follows from either a direct argument or the fact that the special case in which $\Delta_i=0$ for all $i$ has to recover the classical case of Manin's conjecture.

Undoing the factorization from \cref{lem:mobius-factorization} to apply~\cite[Lem.~11.15~(e)]{MR1679841} (with $f=\rho\ge 2$), the series
\begin{equation*}
	\sum_{\vec\eta}c_{\Pc} 2^{-\rho} \mu'(\vec\eta)\prod_{d\in \Dc} \frac{c_da_d^{\Delta_d}}{\eta_d^{a_d}} B (\log B)^{\rho-1-\Delta}
\end{equation*}
converges absolutely to
\begin{equation*}
	c B (\log B)^{\rho-1-\Delta},
\end{equation*}
where the leading constant 
\begin{equation*}
	c= c_{\cal P} 2^{-\rho} \tau \prod_{d\in \Dc} c_d a_d^{\Delta_d}
\end{equation*}
contains the absolutely convergent product
\begin{equation*}
	\tau =
	\prod_p \sum_{\vec\eta}^{(p)}
	\frac{\mu'(\vec\eta)}{\prod_{d\in \Dc} \eta_d} = \prod_p \frac{|\mathfrak{Y}(\FF_p)|}{p^{n+\rho}} = \prod_p \left( 1-\frac{1}{p}\right)^{\rho}\frac{|\mathfrak{X}(\FF_p)|}{p^n}.
\end{equation*}
As usual, the sum should be interpreted to run over powers of $p$; the first equality follows from~\cite[(11.14)]{MR1679841}, and the second one from $|\mathfrak{Y}(\FF_p)| = (p-1)^\rho|\mathfrak{X}(\FF_p)|$, given that $\mathfrak{Y}$ is a $\mathbb{G}_{\mathrm{m}}^\rho$-torsor over $\mathfrak X$.
Once more applying Salberger's lemma \cite[Lem.~11.15~(e)]{MR1679841} (and $\max_{d\in \Dc} (1+\log \eta_d)\ll \prod_{d\in \Dc} \eta_d^{1/4}$) to sum the error term in~\eqref{eq:applied-main-thm}, we arrive at
\begin{equation*}
	N(B) = c B (\log B)^{\rho-1-\Delta} + O(B (\log B)^{\rho-1-\Delta-\delta}).
\end{equation*}
Finally, once more using that this formula has to specialize to the well-known case of Manin's conjecture for rational points on split toric varieties if all $\psi_d$ are identity morphisms, in which case $c_d=2^{a_d}$ and their product is $2^{n+\rho}$, the constant $2^{-\rho}c_{\cal P}$ needs to admit the shape 
\begin{equation*}
	2^{-n-\rho} \alpha(X) \omega_\infty,
\end{equation*}
where $\alpha(X)$ is Peyre's $\alpha$-constant and $\omega_\infty = \tau(X(\RR))$ is the real Tamagawa volume. This completes the proof of \cref{main_theorem_toric}.

\subsection{Examples}\label{section:examples} We conclude by analysing what input conditions make the vector $\vec \ell$ of convergent type for two families of toric varieties: those of Picard rank $2$ and blow-ups of projective space in skew linear subspaces.

\subsubsection*{Picard rank 2}
By a classical fact due to Kleinschmidt~\cite{MR954243}, every toric variety of Picard rank $2$ can be obtained as a projective bundle 
\begin{equation*}
	X = \PP\!\left(\cal O_{\PP^{a_0-1}}^{\oplus a_1} \oplus \cal O_{\PP^{a_0-1}}(d_2)^{\oplus a_2} \oplus \cdots \oplus \cal O_{\PP^{a_0-1}}(d_r)^{\oplus a_r}\right)
\end{equation*}
over a projective space $\PP^{a_0-1}$, where $r\ge 1$, $0=d_1<d_2<\dots<d_r$, and $a_0,\dots,a_r>0$. Write $R = \sum_{i\ge 1} a_i$ and $S= \sum_{i\ge 1}a_id_i$. This type of variety can thus be seen as a higher-dimensional analogue of a Hirzebruch surface.
From the description of its fan~\cite[§\,7.3]{MR2810322}, we may readily obtain its Picard group and the grading of its Cox ring.
The former is freely generated by the pullback $H$ of a hyperplane class and the tautological class $\xi$. The latter is generated by $a_0$ variables of degree $H$ and, for each $1\le i\le r$, a set of $a_i$ variables of degree $\xi-d_iH$. Its anticanonical class
\begin{equation*}
	-K_X = (a_0-S) H + R \xi
\end{equation*}
is ample if and only if $a_0>S$, which we shall assume from now on. Indeed, to check this we may use the description of the nef cone in Huang's work~\cite[Lem.~7.3]{MR4243654}, for instance, where the rays $\rho_{t}$ and $\rho_{n+1}$ in their notation correspond to the variables $x_{r,a_r}$ and $x_{0,a_0}$ in ours, so that $[D_{\rho_t}] = \xi - d_rH$ and $[D_{\rho_{n+1}}]=H$.

Then an anticanonical height function, factorized as in~\eqref{eq:package-height}, is given by
\begin{equation*}
	\max_{1\le i\le r} b_0^{a_0-S+Rd_i} b_i^R,
\end{equation*}
so that the maximal face is cut out by the $r$ equations 
\begin{equation*}
	(a_0-S+Rd_i)t_0 + Rt_i = 1.
\end{equation*}
We solve each of these for $t_i$. The maximal face is then described by the inequalities $0\le t_0\le 1/(a_0+S+Rd_r)$; it is thus the edge between the two vertices
\begin{equation*}
	\vec v_0 = \left(\frac{1}{a_0-S+Rd_r},\ v_{0,1},\ \dots,\ v_{0,r-1},\ 0 \right),
\end{equation*}
where the entry
\begin{equation*}
	v_{0,i} = R^{-1} \left(1-\frac{a_0-S+Rd_i}{a_0-S+Rd_r}\right)
\end{equation*}
is non-zero for $1\le i\le r-1$ and vanishes for $i=r$, and
\begin{equation*}
	\vec v_1 = \left(0,\ R^{-1},\ \dots,\ R^{-1}\right),
\end{equation*}
where all entries except the first one are non-zero.
\cref{prop:convergence-test} with this ordering of the vertices introduces the requirement that $\Delta_r<1$, while applying it to the reverse ordering yields the condition $\Delta_0<1$. There are no restrictions on the remaining $\Delta_i$ (with $1\le i\le r-1$).

In particular, taking $r\ge 3$ and $a_i=3$ for some $1\le i\le r$ makes it possible to use the classical equation $x_{i,1}y_{i,1}^2+x_{i,3}y_{i,3}^2+x_{i,3}y_{i,3}^2=0$ in Serre's problem that was solved by Loughran, Rome, and Sofos~\cite{loughran2024leadingconstantrationalpoints}. In this case, $\Delta_i=3/2$, with a contribution of $1/2$ coming from each of the fibers above $x_{i,1}=0$, $x_{i,2}=0$, and $x_{i,3}=0$. The corresponding divosors on $X$ all have identical class, thus leaving the Loughran--Rome--Sofos setting (that $k[X\setminus D]^\times =k^\times$ is equivalent to saying that all components of $D$ are linearly independent in the Picard group). Still, our asymptotic formula conforms to their order of magnitude as long as $\Delta_0$ and $\Delta_r$ are smaller than $1$.

\subsubsection*{Blow-ups of skew linear subspaces}
A straightforward way to obtain examples of higher Picard rank is via blow-ups. Call a set $Z_1,\dots, Z_r$ of torus-invariant linear subspaces of projective space $\PP^{n-1}$ \emph{skew} if no two of them are contained in the same torus-invariant divisor. Up to permutations of the coordinate variables, such a configuration is determined by fixing
$r\ge 1$, $a_0\ge 1$, and $a_1,\dots,a_r\ge 2$, where $n=a_0+\dots+a_r$, and $Z_i$ is of codimension $a_i$. 
The blow-up $X$ of $\PP^{n-1}$ in such subspaces is again toric.
Its Picard group is freely generated by the pullback $D_0=H$ of a hyperplane class and the exceptional classes $E_1,\dots, E_r$.
Its Cox ring is generated by variables $x_{i,1},\dots,x_{i,a_i}$ of degree $H$ (if $i=0$) and $H-E_i$ (if $i\ne 0$), respectively, together with variables $y_1,\dots,y_r$, where $y_i$ is of degree $E_i$.
These coordinates correspond to strict transforms of divisors containing none of the centres and of divisors containing precisely the centre $Z_i$, and to the exceptional divisors, respectively.
In particular, $\Dc = \{H,H-E_1,\dots, H-E_r, E_1,\dots, E_r\}$.
As every fibration involving only a $y_i$ for $1\le i\le r$ is trivial, we may restrict ourselves to families parametrized by the former variables, that is, $\Ic = \{H,H-E_1,\dots, H-E_r\}$ or a subset thereof.
Thus, let $\psi_0,\dots,\psi_r$ be fibrations in the first sets of variables, with corresponding logarithmic savings $\Delta_0,\dots,\Delta_r\ge 0$ in their asymptotic formulas~\eqref{eq:affine-asymp}.

The anticanonical class is
\begin{equation*}
	-K_X=nH-\sum_{i=1}^r (a_i-1)E_i,
\end{equation*}
and its corresponding height function in the factorization~\eqref{eq:package-height} is 
\begin{equation*}
	H(\vec b) = \max\Bigg(
		\Big\{ b_0^{a_0+r} \prod_{1\le i\le r} b_i^{a_i-1} \Big\}
		\cup
		\Big\{ 
			b_j^{a_0+a_j+r-1}b_j'^{a_j+r} \prod_{\substack{1\le i\le r\\i\ne j}} b_i^{a_i-1} : 1\le j\le r
		\Big\}
	\Bigg),
\end{equation*}
where $b_i$ corresponds to the variables $\{x_{i,1},\dots,x_{i,a_i}\}$ for $0\le i\le r$ and $b_i'$ corresponds to $y_i$ for $1\le i\le r$.

In particular, the maximal face is of codimension $r+1$ in $\RR_{\ge 0}^{2r+1}$ (with variables $t_0,\dots,t_r,t_1',\dots,t_r'$) and cut out by the $r+1$ linearly independent relations
\begin{align*}
	(a_0+r)t_0 + \sum_{1\le i\le r} (a_i-1)t_i  &= 1  \text { and}\\
	(a_j+r)t_j' + (a_0+a_j+r-1)t_j + \sum_{\substack{1\le i\le r\\ i\ne j}} (a_i-1)t_i  &=1 \text{ for } 1\le j\le r.
\end{align*}
We may solve for $t_0$ and $t_1',\dots,t_r'$, respectively, to turn these equalities into inequalities for the independent variables $t_1,\dots,t_r$. The first of these ($\sum(a_i-1)t_i\le 1$) is implied by the others. As there are no conditions and no exponents $\ell_i'$ associated with the last variables, we need only care about the first $r+1$ coordinates of the vertices of the maximal face and readily determine them to be 
\begin{equation*}
	\vec x_0 = \left(\frac{1}{a_0+r},\ 0,\ \dots,\  0\right),\ 
\end{equation*}
together with
\begin{equation*}
	\vec x_1 = \left(*, \frac{1}{a_0+a_1+r-1}, 0, \dots, 0\right),\ \dots,\ \vec x_r = \left(*, 0,\dots, 0, \frac{1}{a_0+a_r+r-1}\right),
\end{equation*}
where the omitted first entry
\begin{equation*}
	\frac{1}{a_0+r}\left(1- \frac{a_i-1}{a_0+a_i+r-1}\right)
\end{equation*}
is non-zero in each case.

We proceed by applying \cref{prop:convergence-test}.
Given an ordering $\vec v(0) = \vec x_0$ and $\vec v(1) = \vec x_j$ with $1\le j\le r$, the constraint for $i=1$ is exactly $\Delta_j < 1$. Since no vertex is zero in the first coordinate, $\Delta_0$ does not appear in any constraint. Given any ordering $\vec v \colon \{0,\ldots,r\} \to \cal E$, at most $i$ choices of $j$ satisfy $s(\vec v)_{j} \leq i$. Hence, the $i$-th condition is already implied by $\Delta_{j'} < 1$ for $1 \leq j' \leq r$.

The total necessary and sufficient condition is thus that $\Delta_i<1$ for $i\ne 0$, while $\Delta_0$ may be arbitrary.

\bibliographystyle{plain}
\bibliography{references.bib}

@article{BloBru18,
	author = {Blomer, V. and Br{\"u}dern, J.},
	title = {Counting in hyperbolic spikes: {The} {Diophantine} analysis of multihomogeneous diagonal equations},
	fjournal = {Journal f{\"u}r die Reine und Angewandte Mathematik},
	journal = {J. Reine Angew. Math.},
	issn = {0075-4102},
	volume = {737},
	pages = {255--300},
	year = {2018},
	language = {English},
}

@article{PieSch24,
	author = {Pieropan, M. and Schindler, D.},
	title = {Hyperbola method on toric varieties},
	fjournal = {Journal de l'{\'E}cole Polytechnique -- Math{\'e}matiques},
	journal = {J. {\'E}c. Polytech., Math.},
	issn = {2429-7100},
	volume = {11},
	pages = {107--157},
	year = {2024},
	language = {English},
}

@article{BCFJK16,
	author = {Bhargava, M. and Cremona, J. E. and Fisher, T. and Jones, N. G. and Keating, J. P.},
	title = {What is the probability that a random integral quadratic form in {$n$} variables has an integral zero?},
	journal = {Int. Math. Res. Not. IMRN},
	fjournal = {International Mathematics Research Notices. IMRN},
	year = {2016},
	number = {12},
	pages = {3828--3848},
}

@unpublished{CKR25,
	title = {Serre's problem for multiple conics},
	author = {S. Chan and P. Koymans and N. Rome},
	year = {2025},
	note = {arXiv preprint (2504.21792)},
}

@unpublished{DLS25,
	title = {Local solubility in generalised {C}h\^atelet varieties},
	author = {K. Destagnol and J. Lyczak and E. Sofos},
	year = {2025},
	note = {arXiv preprint (2504.11388)},
}

@unpublished{GGSW26,
	title = {On indefinite integral ternary quadratic forms},
	author = {A. Gamburd and A. Ghosh and P. Sarnak and J. P. Whang},
	year = {2026},
	note = {arXiv preprint (2603.05849)},
}

@unpublished{Wilson,
	title = {Asymptotics for local solubility of diagonal quadrics over a split quadric surface},
	author = {C. Wilson},
	year = {2024},
	note = {arXiv preprint (2404.11489)},
}

@unpublished{daSilva,
	title = {Solubility of a family of conics with polynomial coefficients in
	         many variables},
	author = {M. Da Silva},
	year = {2025},
	note = {arXiv preprint (2511.20282v2)},
}

@article{Guo,
	author = {Guo, C. R.},
	title = {On solvability of ternary quadratic forms},
	fjournal = {Proceedings of the London Mathematical Society. Third Series},
	journal = {Proc. Lond. Math. Soc. (3)},
	issn = {0024-6115},
	volume = {70},
	number = {2},
	pages = {241--263},
	year = {1995},
	language = {English},
}

@article{Hooley93,
	author = {Hooley, C.},
	title = {On ternary quadratic forms that represent zero},
	fjournal = {Glasgow Mathematical Journal},
	journal = {Glasg. Math. J.},
	issn = {0017-0895},
	volume = {35},
	number = {1},
	pages = {13--23},
	year = {1993},
	language = {English},
}

@article{FHP21,
	author = {Fisher, T. and Ho, W. and Park, J.},
	title = {Everywhere local solubility for hypersurfaces in products of projective spaces},
	fjournal = {Research in Number Theory},
	journal = {Res. Number Theory},
	issn = {2522-0160},
	volume = {7},
	number = {1},
	pages = {27},
	note = {Id/No 6},
	year = {2021},
	language = {English},
}

@unpublished{Schaefer,
	title = {Sums with hyperbolic conditions},
	author = {L. Sch\"{a}fer},
	year = {2025},
	note = {PhD thesis, G\"{o}ttingen University},
}

@unpublished{KPSS25,
	title = {Local solubility of generalised {F}ermat equations},
	author = {P. Koymans and R. Paterson and T. Santens and A. Shute},
	year = {2025},
	note = {arXiv preprint (2501.17619)},
}

@article{LS16,
	author = {Loughran, D. and Smeets, A.},
	title = {Fibrations with few rational points},
	journal = {Geom. Funct. Anal.},
	fjournal = {Geometric and Functional Analysis},
	volume = {26},
	year = {2016},
	number = {5},
	pages = {1449--1482},
}

@article{Lou18,
	author = {Loughran, D.},
	title = {The number of varieties in a family which contain a rational point},
	journal = {J. Eur. Math. Soc. (JEMS)},
	fjournal = {Journal of the European Mathematical Society (JEMS)},
	volume = {20},
	year = {2018},
	number = {10},
	pages = {2539--2588},
}

@unpublished{loughran2024leadingconstantrationalpoints,
	title = {The leading constant for rational points in families},
	author = {D. Loughran and N. Rome and E. Sofos},
	year = {2024},
	note = {arXiv preprint (2210.13559)},
}

@article{MS22,
	author = {Mitankin, V. and Salgado, C.},
	title = {Rational points on del {P}ezzo surfaces of degree four},
	journal = {Int. J. Number Theory},
	fjournal = {International Journal of Number Theory},
	volume = {18},
	year = {2022},
	number = {9},
	pages = {2099--2127},
}

@article{MR1679841,
	author = {Salberger, P.},
	title = {Tamagawa measures on universal torsors and points of bounded height on {F}ano varieties},
	note = {Nombre et r\'epartition de points de hauteur born\'ee (Paris, 1996)},
	journal = {Ast\'erisque},
	fjournal = {Ast\'erisque},
	volume = {251},
	year = {1998},
	pages = {91--258},
	issn = {0303-1179},
	mrclass = {11G50 (11G35 14G05 14G25 14J45)},
	mrnumber = {1679841},
	mrreviewer = {David Harari},
}

@article{Sof16,
	author = {Sofos, E.},
	title = {Serre's problem on the density of isotropic fibres in conic bundles},
	journal = {Proc. Lond. Math. Soc. (3)},
	fjournal = {Proceedings of the London Mathematical Society. Third Series},
	volume = {113},
	year = {2016},
	number = {2},
	pages = {261--288},
}

@article{MR4243654,
	author = {Huang, Z.},
	title = {Rational approximations on toric varieties},
	journal = {Algebra Number Theory},
	fjournal = {Algebra \& Number Theory},
	volume = {15},
	year = {2021},
	number = {2},
	pages = {461--512},
	issn = {1937-0652,1944-7833},
	mrclass = {14G05 (11J99 14M25)},
	mrnumber = {4243654},
	doi = {10.2140/ant.2021.15.461},
	url = {https://doi.org/10.2140/ant.2021.15.461},
}

@book{MR2810322,
	author = {Cox, D. A. and Little, J. B. and Schenck, H. K.},
	title = {Toric varieties},
	series = {Graduate Studies in Mathematics},
	volume = {124},
	publisher = {American Mathematical Society, Providence, RI},
	year = {2011},
	pages = {xxiv+841},
	isbn = {978-0-8218-4819-7},
	mrclass = {14M25 (05A15 05E45 52B12)},
	mrnumber = {2810322},
	mrreviewer = {Ivan\ Arzhantsev},
	doi = {10.1090/gsm/124},
	url = {https://doi.org/10.1090/gsm/124},
}

@article{MR954243,
	author = {Kleinschmidt, P.},
	title = {A classification of toric varieties with few generators},
	journal = {Aequationes Math.},
	fjournal = {Aequationes Mathematicae},
	volume = {35},
	year = {1988},
	number = {2-3},
	pages = {254--266},
	issn = {0001-9054,1420-8903},
	mrclass = {14L32 (14J40 52A25)},
	mrnumber = {954243},
	mrreviewer = {G.\ Ewald},
	doi = {10.1007/BF01830946},
	url = {https://doi.org/10.1007/BF01830946},
}

@article{MR4959852,
	author = {Pieropan, M. and Schindler, D.},
	title = {Points of bounded height on certain subvarieties of toric varieties},
	journal = {Algebra Number Theory},
	fjournal = {Algebra \& Number Theory},
	volume = {19},
	year = {2025},
	number = {11},
	pages = {2281--2306},
	issn = {1937-0652,1944-7833},
	mrclass = {11G50 (11A25 11D45 11P21 14G05 14M25)},
	mrnumber = {4959852},
	doi = {10.2140/ant.2025.19.2281},
	url = {https://doi.org/10.2140/ant.2025.19.2281},
}

\end{document}